\documentclass{article}
\usepackage{amsmath,amscd,amstext,amssymb,amsfonts}
\usepackage{amsthm}
\usepackage{graphics}
\usepackage{latexsym,empheq}
\usepackage{color}
\usepackage{bookman}
\usepackage{bm}
\usepackage{cite}

\newcommand{\nd}{\ensuremath {{d}}} %%% dimension variable in this file: d

\newcommand{\beq}{\begin{eqnarray}}
\newcommand{\eeq}{\end{eqnarray}}
\newcommand{\beqq}{\begin{eqnarray*}}
\newcommand{\eeqq}{\end{eqnarray*}}
\newcommand{\bit}{\begin{itemize}}
\newcommand{\eit}{\end{itemize}}
\newcommand{\ds}{\displaystyle}
\newcommand{\R}{\ensuremath{\mathbb R}}\newcommand{\real}{\R}
\newcommand{\Rn}{\ensuremath{{\mathbb R}^{\nd}}} \newcommand{\rn}{\Rn}
\newcommand{\N}{{\ensuremath{\mathbb N}}}\newcommand{\nat}{\N}
\newcommand{\No}{\ensuremath{\N_{0}}}\newcommand{\no}{\N_0}
\newcommand{\Z}{\mathbb Z} 
\newcommand{\Zn}{\Z^{\nd}}
\newcommand{\C}{{\mathbb C}}
\newcommand{\SRn}{\mathcal{S}(\Rn)}
\newcommand{\SpRn}{\mathcal{S}'(\Rn)}

\newcommand{\dint}{\;\mathrm{d}}
\newcommand{\Ft}{\mathcal{F}}
\newcommand{\Fti}{\mathcal{F}^{-1}}
\newcommand{\Swsing}{\ensuremath{\mathbf{S}_{\mathrm{sing}}(w)}}
\newcommand{\Ssing}[1]{\ensuremath{\mathbf{S}_{\mathrm{sing}}(#1)}}

\newcommand{\A}{\ensuremath{A^s_{p,q}}}  
\newcommand{\Ae}{\ensuremath{A^{s_1}_{p_1,q_1}}}  
\newcommand{\Az}{\ensuremath{A^{s_2}_{p_2,q_2}}}  
\newcommand{\B}{\ensuremath{B^s_{p,q}}}  
\newcommand{\be}{\ensuremath{B^{s_1}_{p_1,q_1}}}  
\newcommand{\bz}{\ensuremath{B^{s_2}_{p_2,q_2}}}  
\newcommand{\rae}{\ensuremath{RA^{s_1}_{p_1,q_1}}}  
\newcommand{\raz}{\ensuremath{RA^{s_2}_{p_2,q_2}}}
\newcommand{\rbe}{\ensuremath{RB^{s_1}_{p_1,q_1}}}  
\newcommand{\rbz}{\ensuremath{RB^{s_2}_{p_2,q_2}}}
\newcommand{\F}{\ensuremath{F^s_{p,q}}}  
\newcommand{\fe}{\ensuremath{F^{s_1}_{p_1,q_1}}}  
\newcommand{\fz}{\ensuremath{F^{s_2}_{p_2,q_2}}}  
\newcommand{\Lloc}{\ensuremath{L_1^{\mathrm{loc}}}} 
\def\supp{\mathop{\rm supp}\nolimits}
\newcommand{\bpr}{\begin{proof}}
\newcommand{\epr}{\end{proof}}
\newcommand{\bli}{\begin{list}{}{\labelwidth6mm\leftmargin8mm}}
\newcommand{\eli}{\end{list}}
\newcommand{\wab}{\ensuremath{w_{\alpha,\beta}}}
\newcommand{\wLog}{\ensuremath{w_{(\bm{\alpha},\bm{\beta})}}}

\def\Id{\mathop{\rm Id}\nolimits}
\def\id{\mathop{\rm id}\nolimits}

\newcommand{\idlog}{\ensuremath{\id_B}}

\newcommand{\ignore}[1]{}

\newcommand{\nn}[1]{\ensuremath \nu( #1)}
\newcommand{\tn}{\ensuremath\mathbf{t}}
\newcommand{\codim}{\ensuremath{\mathrm{codim}\,}}

\allowdisplaybreaks[1]
\numberwithin{equation}{section}

\newtheorem{lemma}{Lemma}[section]
\newtheorem{corollary}[lemma]{Corollary}
\newtheorem{proposition}[lemma]{Proposition}
\newtheorem{theorem}[lemma]{Theorem}

\theoremstyle{definition}
\newtheorem{definition}[lemma]{Definition}
\newtheorem{example}[lemma]{Example}
\newtheorem{Exams}[lemma]{Examples}

\theoremstyle{remark}
\newtheorem{remark}[lemma]{Remark}

\title{Nuclear embeddings in weighted function spaces}

\author{Dorothee D. Haroske\footnotemark[1] \ and Leszek Skrzypczak\footnotemark[1]\ \footnotemark[2]}
\date{\today}

\begin{document}
\maketitle

\footnotetext[1]{Both authors were partially supported by the German Research Foundation (DFG), Grant no. Ha 2794/8-1.}
\footnotetext[2]{The author was supported by National Science Center, Poland,  Grant no. 2013/10/A/ST1/00091.}

\begin{abstract}
  We study nuclear embeddings for weighted spaces of Besov and Triebel-Lizorkin
  type where the weight belongs to some Muckenhoupt class and is essentially of polynomial type. Here we can extend our previous results \cite{HaSk-1,HaSk-lim} where we studied the compactness of corresponding embeddings. The concept of nuclearity goes back to Grothendieck who defined it in \cite{grothendieck}. Recently there is a refreshed interest to study such questions \cite{EL-4,EGL-3,Tri-nuclear,CoDoKu,CoEdKu}. This led us to the investigation in the weighted setting. We obtain complete characterisations for the nuclearity of the corresponding embedding. Essential tools are a discretisation in 
  terms of wavelet bases, operator ideal techniques, as well as a very useful result of Tong \cite{tong} about the nuclearity of diagonal operators acting in $\ell_p$ spaces. In that way we can further contribute to the characterisation of nuclear embeddings on domains obtained in \cite{Pie-r-nuc,PiTri,Tri-nuclear,CoDoKu}.\\

\end{abstract}

\section{Introduction}
Grothendieck introduced the concept of nuclearity in \cite{grothendieck} more than 60 years ago. It paved the way to many famous developments in functional analysis later one, like the theories of nuclear locally convex spaces, operator ideals, eigenvalue distributions, and traces and determinants in Banach spaces. Enflo used nuclearity in his famous solution \cite{enflo} of the approximation problem, a long-standing problem of Banach from the Scottish Book. We refer to \cite{Pie-snumb,Pie-op-2}, and, in particular, to \cite{pie-history} for further historic details.

Let $X,Y$ be Banach spaces, $T\in \mathcal{L}(X,Y)$ a linear and bounded operator. Then $T$ is called {\em nuclear}, denoted by $T\in\mathcal{N}(X,Y)$, 
if there exist elements $a_j\in X'$, the dual space of $X$, and $y_j\in Y$, $j\in\mathbb{N}$, such that $\sum_{j=1}^\infty \|a_j\|_{X'} \|y_j\|_Y < \infty$ 
and a nuclear representation $Tx=\sum_{j=1}^\infty a_j(x) y_j$ for any $x\in X$. Together with the {\em nuclear norm}
\[
 \nu(T)=\inf\Big\{ \sum_{j=1}^\infty   \|a_j\|_{X'} \|y_j\|_Y:\ T =\sum_{j=1}^\infty a_j(\cdot) y_j\Big\},
  \]
  where the infimum is taken over all nuclear representations of $T$, the space $\mathcal{N}(X,Y)$ becomes a Banach space. It is obvious that 
	nuclear operators are, in particular, compact.
	
  Already in the early years there was a strong interest to study examples of nuclear operators beyond diagonal operators in $\ell_p$ 
	sequence spaces, where a complete answer was obtained in \cite{tong} (with some partial forerunner in \cite{Pie-op-2}). Concentrating on 
	embedding operators in spaces of Sobolev type, first results can be found, for instance, in \cite{PiTri,Pie-r-nuc,Parfe-2}.

Though the topic was always studied to a certain extent, we realised an increased interest in the last years.  Concentrating on the Sobolev embedding for spaces on a bounded domain, some of the recently published papers we 
have in mind are \cite{EL-4,EGL-3,Tri-nuclear,CoDoKu,CoEdKu} using quite different techniques however.  

We observed several directions and reasons for this. For example, the problem to describe a compact operator outside the Hilbert space setting is a partly
 open and very important one. 
It is well known from the remarkable Enflo result \cite{enflo} that there are compact operators
between Banach spaces which cannot be approximated by finite-rank operators. This led to a number of -- meanwhile well-established 
and famous -- methods to circumvent this difficulty and find alternative ways to `measure' the compactness or `degree' of compactness of an operator. It can be described by the asymptotic behaviour of its approximation or entropy numbers, 
which are basic tools for many different problems nowadays, e.g. eigenvalue distribution of compact operators in Banach spaces, 
optimal approximation of Sobolev-type embeddings, but also for numerical questions. In all these problems, the decomposition
of a given compact operator into a series is an essential proof technique. It turns out that in many of the recent papers  \cite{Tri-nuclear,CoDoKu,CoEdKu} studying nuclearity, a key tool in the arguments are new decomposition techniques as well, adapted to the different spaces. So we intend to follow this strategy, too.

%\ignore{
Concerning weighted spaces of Besov and Sobolev type, we are in some sense devoted to the program proposed by Edmunds and Triebel \cite{ET} to investigate the spectral properties of certain pseudo-differential
operators based on the asymptotic behaviour of entropy and approximation
numbers, together with Carl's inequality and the Birman-Schwinger
principle. Similar questions in the context of weighted function spaces of
this type were studied by the first named
author and {Triebel}, cf. \cite{HT1}, and were continued and extended by
{K{\"u}hn}, {Leopold, Sickel} and the second author in the series of papers
\cite{K-L-S-S-2,K-L-S-S-3,K-L-S-S-4}. Here the considered weights are always assumed to be `admissible': These are smooth weights with no
singular points, with $w(x)=(1+|x|^2)^{\gamma/2}$,
$\gamma\in\R$, $x\in\Rn$, as a prominent example.   

We started in \cite{HaSk-1} a different approach and considered weights from
the Muckenhoupt class $\mathcal{A}_\infty$ which -- unlike `admissible'
weights -- may have local singularities, that can influence embedding 
properties of such function spaces. Weighted Besov and
Triebel-Lizorkin spaces with Muckenhoupt weights are well known concepts,
cf. \cite{bui-1,bui-2,F-Rou,roud,bownik-2,bownik-1,HaPi}. 
In \cite{HaSk-1} we dealt with general transformation methods from function to
appropriate sequence spaces  
provided by a wavelet decomposition; we essentially concentrated on the
example weight 
\[
w_{\alpha,\beta}(x) \sim  \begin{cases}|x|^\alpha & \text{if} \quad  |x|\le
  1\, ,\\ |x|^\beta & \text{if} \quad  |x|> 1\, ,\end{cases} \qquad
\text{with}\quad \alpha>-\nd,\quad \beta>0,
\]
of purely polynomial growth both near the origin and for $|x|\to\infty$. In
the general setting for $w\in\mathcal{A}_\infty$ we
obtained sharp criteria for the compactness of embeddings of
type
\[
\id_{\alpha,\beta} : \Ae(\Rn,\wab) \hookrightarrow \Az(\Rn), 
\]
where $s_2\leq s_1$, $0<p_1, p_2<\infty$, $0<q_1,q_2\leq\infty$, and $\A$
stands for either Besov spaces $\B$ or Triebel-Lizorkin spaces $\F$. More precisely, we proved in \cite{HaSk-1} that $\id_{\alpha,\beta}$ is compact if, and only if,
\[
\frac{\beta}{p_1}>\nd \max\left(\frac{1}{p_2}-\frac{1}{p_1},0\right) \qquad\text{and}\qquad
s_1-\frac{\nd}{p_1}- s_2+\frac{\nd}{p_2} > \max
\left(\nd \max\left(\frac{1}{p_2}-\frac{1}{p_1},0\right),\frac{\alpha}{p_1}\right). \]
In the same paper \cite{HaSk-1} we determined the exact asymptotic behaviour of corresponding entropy  and approximation
numbers of $\id_{\alpha,\beta}$ in the compactness case. Now we can refine this characterisation by our new result about the nuclearity of $\id_{\alpha,\beta}$. One of our main results in the present paper, Theorem~\ref{wab-nuc} below, states that $\id_{\alpha,\beta}$ is nuclear if, and only if, 
\[
\frac{\beta}{p_1}>\nd - \nd \max\left(\frac{1}{p_1}-\frac{1}{p_2},0\right) \qquad\text{and}\qquad
s_1-\frac{\nd}{p_1}- s_2+\frac{\nd}{p_2} > \max
\left(\nd-\nd \max\left(\frac{1}{p_1}-\frac{1}{p_2},0\right),\frac{\alpha}{p_1}\right), \]
where $1\leq p_1<\infty$ and $1\leq p_2, q_1, q_2\leq\infty$. 
In \cite{HaSk-lim} we studied the weight
\[
\wLog(x)\, = \  \begin{cases}
|x|^{\alpha_1} (1-\log|x|)^{\alpha_2}, & \text{if} \quad  |x|\le 1\, ,\\
|x|^{\beta_1} (1+\log|x|)^{\beta_2},& \text{if} \quad  |x|> 1\,,
\end{cases}
\]
where $\bm{\alpha} = (\alpha_1,\alpha_2)$, $\alpha_1>
-\nd$, $\alpha_2\in\R$, $\bm{\beta} = (\beta_1,\beta_2)$, $\beta_1>-\nd$, $\beta_2\in\R$.
Again we obtained the complete characterisation of the compactness of 
\[
\id_{(\bm{\alpha},\bm{\beta})} : \be(\Rn, \wLog) \hookrightarrow  \bz(\Rn),
\]
as well as asymptotic results for the corresponding entropy numbers. The intention was not only to generalise the weight function, but also to cover some limiting cases in that way. Our second main result, Theorem~\ref{wLog-nuc} below, completely answers the question of the nuclearity of $\id_{(\bm{\alpha},\bm{\beta})} $, where now even the fine parameters $q_1, q_2$ are involved in the criterion.

While proving our result we benefit from Tong's observation \cite{tong} (and the fine paper \cite{CoDoKu} which has drawn our attention to it), and the available wavelet decomposition and operator ideal techniques used in our previous papers \cite{HaSk-1,HaSk-lim} already. Moreover, we used and slightly extended Triebel's result \cite{Tri-nuclear} (with forerunners in \cite{PiTri,Pie-r-nuc}) on the nuclearity of the embedding operator
\[
\id_\Omega : \Ae(\Omega) \to \Az(\Omega),
\]
where $\Omega\subset\rn$ is assumed to be a bounded Lipschitz domain and the spaces $\A(\Omega)$ are defined by restriction. In \cite{CoDoKu} some further limiting cases were studied and we may now add a little more to this limiting question.

Beside embeddings of appropriately weighted spaces and embeddings of spaces on bounded domains, we also consider embeddings of radial spaces which may admit compactness,
\[
\id_R: R\Ae(\rn)\to R\Az(\rn),
\]
for definitions we refer to Section~\ref{subsec-rad} below. This has been studied in detail in \cite{SS,SSV}. In particular, we can now gain from the close connection between radial spaces and appropriately weighted spaces established in \cite{SS,SSV}. In that way we are able to prove a criterion of nuclearity of the embedding $\id_R$ in Theorem~\ref{radial-nuc} below.

The paper is organised as follows. In Section~\ref{sect-1} we recall basic facts
about weight classes and weighted function spaces needed later
on. Section~\ref{sect-nuc} is devoted to our main findings about the nuclearity of embeddings: we start with a collection of known results in Section~\ref{subsec-conc-nuc} which we shall need later; in Section~\ref{subsec-wnuc} we present our new results for the weighted embeddings described above, while in Section~\ref{subsec-rad} we turn our attention to radial spaces and nuclearity of embeddings.
%}
%\open{missing}

\section{Weighted function spaces}
\label{sect-1}
First of all we need to fix some notation.
By $\N$ we denote the set of natural numbers, 
by $\No$ the set $\N \cup \{ 0\}$, 
and by $\Zn$ the set of all lattice points
in $\Rn$ having integer components.  

The positive part
of a real function $f$ is given 
by $f_{+}(x)=\max(f(x),0)$. For two positive real sequences 
$\{a_k\}_{k\in \N}$ and $\{b_k\}_{k\in \N}$ we mean by 
$a_k\sim b_k$ that there exist constants $c_1,c_2>0$ such that $c_1\ a_k\leq
b_k\leq c_2 \ a_k$ for all $k\in \N$; similarly for positive
functions.   

Given two (quasi-) Banach spaces $X$ and $Y$, we write $X\hookrightarrow Y$
if $X\subset Y$ and the natural embedding of $X$ in $Y$ is continuous. 

All unimportant positive constants will be denoted by $c$, occasionally with
subscripts. For convenience, let both $\ \dint x\ $ and $\ |\cdot|\ $ stand for the
($\nd$-dimensional) Lebesgue measure in the sequel.
%If not otherwise indicated, $\log$ is always taken with respect to base $2$. 

\subsection{Weight functions}

We shall essentially deal with weight functions of polynomial type. Here we use our preceding results in \cite{HaSk-1,HaSk-gen,HaSk-lim} which partly rely on general features of Muckenhoupt weights. For that reason we first recall some fundamentals on this special class of weights. By a weight $w$ we shall always mean a locally integrable function $w\in \Lloc(\Rn)$, positive a.e. in the sequel. Let $\ M\ $ stand for the
Hardy-Littlewood maximal operator given by  
\begin{equation}\label{M}
  Mf(x)= \sup_{B(x,r) \in \mathcal{B}} \ \frac{1}{|B(x,r)|}\int_{B(x,r)} |f(y)|
  \dint y, \quad x\in\Rn,
\end{equation}
where $\mathcal{B}$ is the collection of all open balls
%\[
$B(x,r)=\Big\{y\in \Rn:\;\; |y-x|<r \Big\}, \quad r>0$.
%\]
 
\begin{definition}
Let $w$ be a weight function on $\rn$.
\bli
\item[{\bfseries (i)}] 
Let $1<p<\infty$. Then $w$ belongs to the Muckenhoupt class
$\mathcal{A}_p$, if there exists a constant $0<A<\infty$ such that for all
balls $B$ the following inequality holds 
\begin{equation}\label{A_p}
\left(\frac{1}{|B|}\int_{B}w(x) \dint x\right)^{1/p} 
\left(\frac{1}{|B|}\int_{B}w(x)^{-p'/p} \dint x \right)^{1/p'} 
\leq A ,  
\end{equation}
where $p'$ is the dual exponent to $p$ given by $1/p'+ 1/p=1$ and $|B|$ stands
for the Lebesgue measure of the ball $B$. 
\item[{\bfseries (ii)}] 
Let $p=1$. Then $w$ belongs to the Muckenhoupt class $\mathcal{A}_1$ if 
there exists a constant $0<A<\infty$ such that the inequality 
\[Mw(x)\leq A w(x) \]
holds for almost all $x\in \rn$.
\item[{\bfseries (iii)}] 
 The Muckenhoupt
class $\mathcal{A}_{\infty}$  is given by  $\quad \ds \mathcal{A}_{\infty}=\bigcup_{p> 1} \mathcal{A}_p$.
\eli
\end{definition}

Since the pioneering work of Muckenhoupt \cite{muck-2,muck,muck-3}, these
classes of weight functions have been studied in great 
detail, we refer, in particular, to the monographs \cite{GC-RdF},
\cite{St-To}, \cite[Ch.~IX]{torchinsky}, 
and \cite[Ch.~V]{stein} for a complete account on the theory of Muckenhoupt
weights. As usual, we use the abbreviation 
\begin{equation}
w(\Omega)=\int_{\Omega}w(x) \dint x,
\label{wOmega}
\end{equation}
where $\Omega\subset\Rn$ is some bounded, measurable set.

\begin{Exams}\label{Ex-Muck}
\begin{itemize}
\item[{\upshape\bfseries (i)}]
%{\upshape
One of the most prominent examples of a Muckenhoupt weight 
$w\in\mathcal{A}_r$, $1\leq r<\infty$, is given by $\ w_\alpha(x) = |x|^\alpha$, where
$w_\alpha \in \mathcal{A}_r$ if, and only if, $-\nd<\alpha <\nd(r-1)$
for $1<r<\infty$, and $-\nd<\alpha \leq 0$ for $r=1$. We modified this example in \cite{HaSk-1} by 
\begin{align}
\wab(x)& = \begin{cases} |x|^\alpha ,& |x|<1, \\ |x|^\beta ,& |x|\geq
1,\end{cases}\label{wab}
\end{align}
where $\alpha, \beta>-\nd$. Straightforward calculation 
shows that for $1<r<\infty$,  $ \wab \in \mathcal{A}_r \quad\text{if, and only if,}\quad 
-\nd<\alpha,\beta<\nd(r-1)$. 
\item[{\upshape\bfseries (ii)}]
We also need the example considered in \cite{HaSk-lim},
\begin{equation}\label{wlog-1}
\wLog(x)\, = \  \begin{cases}
|x|^{\alpha_1} (1-\log|x|)^{\alpha_2}, & \text{if} \quad  |x|\le 1\, ,\\
|x|^{\beta_1} (1+\log|x|)^{\beta_2},& \text{if} \quad  |x|> 1\,,
\end{cases}
\end{equation}
where
\begin{equation}\label{wlog-2}
 \bm{\alpha} = (\alpha_1,\alpha_2), \ \alpha_1>
-\nd,\ \alpha_2\in\R, \quad \bm{\beta} = (\beta_1,\beta_2), \ \beta_1>-\nd, \ \beta_2\in\R.
\end{equation}
A special case here is  the `purely logarithmic' weight
\begin{equation}\label{pure-log-w}
  w^{\log}_{\bm{\gamma}}(x) = \begin{cases}
(1-\log|x|)^{\gamma_1}, & \text{if} \quad  |x|\le 1\, ,\\
(1+\log|x|)^{\gamma_2},& \text{if} \quad  |x|> 1\,,
\end{cases}
\end{equation}
where $\bm{\gamma}=(\gamma_1,\gamma_2)\in\real^2$. Then $w^{\log}_{\bm{\gamma}} \in \mathcal{A}_1$ for $\gamma_2\leq 0\leq \gamma_1$.
\end{itemize}
%%%%%%%%%%%%%%%%%%%%%%%%%%%
\noindent For further examples
we refer to \cite{Fa-So,HaSk-1,HaSk-gen}. 
\label{Ap-exm} 
\end{Exams}
%%%%%%%%%%%%%%%%%%%%%%%%%%%%%%5
We need some refined study of the singularity behaviour of Muckenhoupt 
$\mathcal{A}_\infty$ weights. Let for $m\in \Zn$ and $j \in \no$, $Q_{j,
  m}$ denote the $\nd$-dimensional cube with sides parallel to the axes of
coordinates, centered at $2^{-j}m$ and with side length $2^{-j}$.
In \cite{HaSk-gen} we introduced the following notion of
their {\em set of singularities} $\Swsing$.

\begin{definition}
For $w\in\mathcal{A}_\infty$ we define the {\em set of singularities}
$\Swsing$ by
\begin{align*}
\Swsing & = \ \left\{x_0\in\rn: \ \inf_{Q_{j,m}\ni x_0} \
  \frac{w(Q_{j,m})}{|Q_{j,m}|} = 0 \right\} \ \cup\ \left\{x_0\in\rn: \ \sup_{Q_{j,m}\ni x_0} \
  \frac{w(Q_{j,m})}{|Q_{j,m}|} = \infty \right\}.%\\
\end{align*}
\end{definition}

Recall the following result.

\begin{proposition}[{\cite[Prop.~2.6]{HaSk-morrey-w}}]\label{meas-sing}
If $w\in \mathcal{A}_\infty$,  then $|\overline{\Swsing}|=0$. 
\end{proposition}

\begin{remark}\label{rem-regular-w}
$\Swsing$ is a special case of $\Ssing{w_1,w_2}$ defined in \cite{HaSk-gen} with
$w_2\equiv 1$, $w_1\equiv w$. There we also  proved some forerunner of Proposition~\ref{meas-sing}. Let us explicitly recall a very useful consequence of the above result, cf. \cite[Cor.~2.7]{HaSk-morrey-w}. 
We call a cube (or ball) $Q\subset\rn$  {\em regularity cube} (or
{\em regularity ball}) of a given weight $w$, if the weight is regular
there, that is, if there exist positive constants $c_1, c_2$ such that
for all $x\in Q$ it holds $c_1\leq w(x)\leq c_2$, i.e.,
$w\sim 1$ on $Q$. Hence the above proposition implies that for any $w\in \mathcal{A}_\infty$ any cube or ball $Q\subset\rn$ contains a regularity cube or ball $\widetilde{Q}\subset Q$.
\end{remark}

\begin{remark}\label{weights-adm}
In \cite{HT1} we studied so-called `admissible' weights. These are smooth weights with no singular points. One can take 
\[ 
w(x)=\langle x\rangle^\gamma = (1+|x|^2)^{\gamma/2}, \quad \gamma\in\real,\quad x\in\rn,
\]
as a prominent example. For the precise definition we refer to \cite{HT1} and the references given therein.
\end{remark}

\subsection{Weighted function spaces of type $\B(\rn,w)$ and
  $\F(\rn,w)$} 

Let $w\in \mathcal{A}_{\infty}$ be a Muckenhoupt weight, and $\ 0<p<
\infty$. Then the weighted Lebesgue space $L_p(\Rn, w)$ contains all
measurable functions such that  
\begin{equation}\label{Lebesqueweighted}
\|f\;|L_p(\Rn,w)\|= \left( \int_{\Rn}|f(x)|^p w(x) \dint x\right)^{1/p}
\end{equation}
is finite. Note that for $p=\infty$ one obtains the classical (unweighted)
Lebesgue space,
\begin{equation}
L_\infty(\Rn, w) = L_\infty(\Rn),\quad w\in\mathcal{A}_\infty.
\label{infty-w}
\end{equation}
Thus we mainly restrict ourselves to $\ p<\infty$ in what follows. \\

The Schwartz space $\ \SRn\ $ and its dual $\ \SpRn\ $ of all 
complex-valued tempered distributions have their usual meaning here.
Let  $\ \varphi_0=\varphi \in \SRn\ $ be such that  
\begin{equation*}
\supp \varphi\subset\left\{y\in\Rn:|y|<2\right\}\quad \mbox{and}\quad
\varphi(x)=1\quad\mbox{if}\quad |x|\leq 1\;,
\end{equation*}
and for each $\ j\in\N\;$ let $\ \varphi_j(x)=
\varphi(2^{-j}x)-\varphi(2^{-j+1}x)$. Then $\ \{\varphi_j\}_{j=0}^\infty $
forms a {\em smooth dyadic resolution of unity}. Given any $\ f\in \SpRn$, we
denote by $\Ft f$ and $\Fti f$ its Fourier transform and its
inverse Fourier transform, respectively.

\begin{definition}
Let $0 <q\leq \infty$, $\ 0<p<\infty$, $\ s\in \mathbb{R}$ and $\ \left\{\varphi_j\right\}_j$
a smooth dyadic resolution of unity. Assume $\ w\in\mathcal{A}_\infty$.
\bit
\item[{\upshape\bf (i)}] {\em The weighted Besov  space}
  $B_{p,q}^{s}(\rn,w)$ is the set of all distributions $f\in \SpRn$ such that 
\begin{align}\label{Bw}
\big\|f\;|B_{p,q}^{s}(\rn,w)\big\|=
\left\| \left\{ 2^{js}\big\|\mathcal{F}^{-1}(\varphi_j\mathcal{F}f)|
L_p(\rn,w)\big\|\right\}_{j\in\no} | \ell_q\right\| 
\end{align}
is finite. 
\item[{\upshape\bf (ii)}] {\em The weighted Triebel - Lizorkin  space} $F_{p,q}^{s}(\rn,w)$  
is the set of all distributions $f\in   \SpRn$ such that 
\begin{align}\label{Fw}
\big\|f\;|F_{p,q}^{s}(\rn,w)\big\|=\left\| \big\| \left\{2^{js}|
 \mathcal{F}^{-1}(\varphi_j\mathcal{F}f)(\cdot)|\right\}_{j\in\no} |
 \ell_q\big\|~ |L_p(\rn,w)\right\| 
\end{align} 
is finite. 
\eit
\end{definition}

\begin{remark}
The  spaces $B_{p,q}^s(\rn,w)$ and $F_{p,q}^s(\rn,w)$ are
independent of the particular choice of the smooth dyadic resolution of unity
$\left\{\varphi_j\right\}_j $ appearing in their definitions. They are quasi-Banach spaces
(Banach spaces for $p,q\geq 1$), and $\ \SRn \hookrightarrow \B(\Rn,
w)\hookrightarrow \SpRn$, similarly for the $F$-case, where the
first embedding is dense if $q<\infty$; cf. \cite{bui-1}. Moreover, for $\
w_0\equiv 1 \in \mathcal{A}_\infty$ 
we obtain the usual (unweighted) Besov and Triebel-Lizorkin spaces; we
refer, in particular, to the series of monographs by {Triebel} 
\cite{T-F1,T-F2,T-Frac,T-F3} for a 
comprehensive treatment of the unweighted spaces. \\
The above spaces with weights of type $\ w\in \mathcal{A}_\infty$ have been
studied systematically by {Bui} first in \cite{bui-1,bui-2}. It turned out that many of the
results from the unweighted situation have weighted counterparts: e.g., we
have $F^0_{p,2}(\Rn,w) = h_p(\Rn,w)$, $\ 0<p<\infty$, where  the latter are
Hardy spaces, see \cite[Thm.~1.4]{bui-1}, and, in particular, $h_p(\Rn, w) =
L_p(\Rn, w) = F^0_{p,2}(\Rn,w)$, $\ 1<p<\infty$, $\ w\in \mathcal{A}_p$, see \cite[Ch.~VI,
Thm.~1]{St-To}. Concerning (classical) Sobolev spaces $W^k_p(\Rn, w) $ built
upon $L_p(\Rn, w)$ in the usual way, it holds 
\beq
W^k_p(\Rn, w) = F^k_{p,2}(\Rn,
w), \quad k\in\No, \quad 1<p<\infty, \quad w\in \mathcal{A}_p,
\label{W=F}
\eeq
cf. \cite[Thm.~2.8]{bui-1}. 
%Further results can be found in \cite{bui-1,bui-2,GC-RdF,roud-2}. 
In \cite{rychkov-5} the above class of weights was extended to the class $\
\mathcal{A}_p^{\mathrm{loc}}$. %Recent works are due to {Roudenko} \cite{roud,roud-2,F-Rou} and {Bownik}\cite{bownik-2,bownik-1}.
We partly rely on our approaches \cite{HaPi,HaSk-1,HaSk-gen}.
\end{remark}

{\em Convention}. We adopt the nowadays usual custom to write $\A$ instead of $\B$ or $\F$, respectively,
when both scales of spaces are meant simultaneously in some context (but
always with the understanding of the same choice within one and the same
embedding, if not otherwise stated explicitly).

\begin{remark}%\label{Aw-emb}
  Occasionally we use the following embeddings which are natural extensions from the unweighted case.
If $0<q\leq\infty$,  $0<q_0\leq q_1\leq\infty$, $\ 0<p<\infty$, $\ s, s_0,s_1\in \mathbb{R}$ with $s_1\leq s_0$, and $\ w\in\mathcal{A}_\infty$, then $A^{s_0}_{p,q}(\Rn, w) \hookrightarrow
A^{s_1}_{p,q}(\Rn,w)$ and $A^{s}_{p,q_0}(\Rn, w) \hookrightarrow A^{s}_{p,q_1}(\Rn,
w)$, and
\beq
B^s_{p, \min(p,q)}(\Rn, w) \hookrightarrow \F(\Rn, w) \hookrightarrow B^s_{p,
  \max(p,q)}(\Rn, w). 
\label{B-F-B}
\eeq
For the unweighted case $w\equiv 1$ see \cite[Prop. 2.3.2/2, Thm. 2.7.1]{T-F1} and
\cite[Thm. 3.2.1]{SiT}. The above result essentially coincides with
\cite[Thm.~2.6]{bui-1} and can be found in \cite[Prop.~1.8]{HaSk-1}. 
\end{remark}

Finally, we briefly describe the wavelet characterisations of  Besov spaces with
${\mathcal A}_\infty$ weights proved in \cite{HaSk-1}. 
Let for $m\in \Zn$ and $j \in \No$ the cubes $Q_{j, m}$ be as above. Apart from function
spaces with weights we introduce sequence  
spaces with weights: for $0<p<\infty$, $0<q  \leq\infty$, $\sigma\in \R$, and
$w\in \mathcal{A}_\infty$, let
\begin{align*}
b^\sigma_{p,q}(w)   :=  & \ \Bigg\{ \lambda = 
\{\lambda_{j,m}\}_{j\in\no, m\in\Zn} : \     \lambda_{j,m} \in \C\, , 
\\
&~\qquad \| \, \lambda \, |b^\sigma_{p,q}(w)\| \sim \Big\| 
\Big\{2^{j \sigma }\,  \Big(\sum_{m \in
  \Zn} |\lambda_{j,m}|^p\ 2^{j \nd}\ w(Q_{j,m})\Big)^{\frac1p}
\Big\}_{j\in\no} | \ell_q\Big\| < \infty \Bigg\}\,  
\intertext{and}
\ell_p(w)  :=  & \ \Bigg\{ \lambda = 
\{\lambda_m\}_{m\in\Zn} :      \lambda_{m} \in \C, \  
\| \, \lambda \, |\ell_p(w)\| \sim  \Big(\sum_{m \in \Zn} |\lambda_{m}|^p
\ 2^{j \nd}\ w(Q_{0,m})\Big)^{\frac1p}  < \infty \Bigg\}\, .
\end{align*}
If $w\equiv 1$ we write $b^\sigma_{p,q}$ instead of
$b^\sigma_{p,q}(w) $.  

Let $\widetilde{\phi}\in C^{N_1}(\R)$ be a scaling function 
on $\R$ with $\supp \widetilde{\phi}\subset [-N_2,\, N_2]$ for certain natural
numbers $N_1$ and $N_2$, and $\widetilde{\psi}$ an associated wavelet. Then the 
tensor-product ansatz yields a scaling function $\phi$  and associated wavelets
$\psi_1, \ldots, \psi_{2^{\nd}-1}$, all defined now on $\Rn$. 
This implies
\beq\label{2-1-2}
\phi, \, \psi_i \in C^{N_1}(\Rn) \quad \text{and} \quad 
\supp \phi ,\, \supp \psi_i \subset [-N_3,\, N_3]^{\nd} \, , 
\quad i=1, \ldots \, , 2^{\nd}-1 \, .
\eeq
Using the standard abbreviations  $\phi_{j,m}(x) =  2^{j \nd/2} \,
\phi(2^j x-m)$ and $\psi_{i,j,m}(x) =  2^{j \nd/2} \, \psi_i(2^j x-m)$
we proved in \cite{HaSk-1} the following wavelet decomposition result.

\begin{theorem}[{\cite[Thm.~1.13]{HaSk-1}}]\label{waveweight}  
Let $0 < p,q \le \infty$ and let $s\in \R$. 
Let $\phi$ be a scaling function and let $\psi_i$,  $i=1, \ldots ,2^{\nd}-1$, be
the corresponding wavelets satisfying \eqref{2-1-2}. We assume that
$|s|<N_1$. Then a distribution $f \in \SpRn$ belongs to $\B(\Rn,w)$,
if, and only if, 
\[
\| \, f \, |\B(\Rn, w)\|^\star  =  
\Big\| \left\{\langle f,\phi_{0,m}\rangle \right\}_{m\in \Zn} | \ell_p(w)\Big\| + \sum_{i=1}^{2^{\nd}-1}
\Big\| \left\{\langle f,\psi_{i,j,m}\rangle \right\}_{j\in \No, m\in \Zn} | b^\sigma_{p,q}(w)\Big\|
\]
is finite, where $\sigma=s+\frac \nd 2 - \frac \nd p$. Furthermore,
$\| \, f \, |\B(\Rn, w) \|^\star $ may be used as an 
equivalent $($quasi-$)$ norm in 
$\B(\Rn, w)$.
\end{theorem}

%\open{radial spaces, connection to weighted spaces, for further use}

\subsection{Compact embeddings}
\label{sect-2}

We collect some compact embedding results for weighted spaces
of the above type that will be used later. For that purpose, let us introduce the following notation: for $s_i\in\real$, $0<p_i,q_i\leq\infty$, $i=1,2$, we call 
\begin{equation}
\delta:=s_1 - \frac{\nd}{p_1}-s_2 + \frac{\nd}{p_2}, \quad 
\label{delta}
\end{equation}
and 
\begin{equation}\label{pq-star}
  \frac{1}{p^*} = \max\left( \frac{1}{p_2} - \frac{1}{p_1},0\right),
  \quad \frac{1}{q^*} = \max\left( \frac{1}{q_2} - \frac{1}{q_1},0\right)
  \end{equation}
(with the understanding that $p^\ast=\infty$ when $p_1\leq p_2$, $q^\ast=\infty$ when $q_1\leq q_2$).

We restrict ourselves to the situation when
  only the source space is weighted, and the target space
  unweighted,
\begin{equation}
\Ae(\Rn,w)\hookrightarrow \Az(\Rn),
\label{emb-Bw}
\end{equation}
where $w\in\mathcal{A}_\infty$. The weight we now consider is either $\wab$ given by \eqref{wab}, or $\wLog$ given by \eqref{wlog-1} (with the special case $w^{\log}_{\bm{\gamma}}$ as in \eqref{pure-log-w}). 
Moreover, we shall assume in the sequel that $p_1<\infty$ for
  convenience, as otherwise we have $\be(\Rn,
  w)=\be(\Rn)$, recall \eqref{infty-w}, and we arrive at the
  unweighted situation in \eqref{emb-Bw} which is
  well-known already. 

 We first recall the result for 
  Example~\ref{Ex-Muck}(i).

\begin{proposition}[{\cite[Prop.~2.6]{HaSk-1}}]\label{emb2}
  Let $\alpha>-\nd$, $\beta>-\nd$, $\wab$ be given by \eqref{wab} and
\begin{equation}
-\infty < s_2\leq s_1<\infty, \quad 0<p_1<\infty,\quad 0<p_2\leq\infty, \quad
 0<q_1, q_2\leq\infty.
\label{param}
\end{equation} 
  Then the embedding
  \[\id_{\alpha,\beta}: \Ae(\Rn,\wab)\hookrightarrow
  \Az(\Rn)\]
  is compact if, and only if,  
\beq
\frac{\beta}{p_1}>\frac{\nd}{p^*} \qquad\text{and}\qquad \delta >\max
\left(\frac{\nd}{p^*},\frac{\alpha}{p_1}\right). 
\label{dd-6}
\eeq
%\eli
\end{proposition}

\begin{remark}\label{rem-wave-red}
  Let us briefly point out the main argument in \cite{HaSk-1,HaSk-gen,HaSk-lim} concerning compactness assertions as we shall follow a similar idea when dealing with nuclearity below. We rely on a reduction of the function space embeddings to corresponding sequence space embeddings based on
    the wavelet decomposition Theorem~\ref{waveweight}:  we make use of the commutative diagram
$$
\begin{array}{ccc}\be(\Rn,w_1) & \xrightleftharpoons[T^{-1}]{T} &
  b^{\sigma_1}_{p_1,q_1}(w_1) \\ \Id \Big\downarrow & & \Big\downarrow \id\\
\bz(\Rn,w_2)& \xleftrightharpoons[S^{-1}]{S} &b^{\sigma_2}_{p_2,q_2}(w_2)
\end{array}
$$
with appropriate isomorphisms $S$ and $T$. Similarly, with 
an appropriate isomorphism $A$  it is sufficient to investigate the embedding of a weighted sequence space into an unweighted one, using
$$
\begin{array}{ccc}
b^{\sigma_1}_{p_1,q_1}(w_1)& \xrightleftharpoons[A^{-1}]{A} & b^{\sigma_1}_{p_1,q_1}(w_1/w_2)\\
\Id \Big\downarrow & & \Big\downarrow \id\\
b^{\sigma_2}_{p_2,q_2}(w_2)& \xleftrightharpoons[A]{A^{-1}} & b^{\sigma_2}_{p_2,q_2}\, 
\end{array}
$$
This will be our starting point below.
\end{remark}

\begin{remark}\label{rem-adm-comp}
  In the special case $\alpha=0$ the weight $w_{0,\beta}$ can be regarded as a so-called admissible weight, $w_{0,\beta}(x) \sim \langle x\rangle^\beta =:w^\beta(x)$, recall Remark~\ref{weights-adm}. For such weights compact embeddings were studied in many papers, see for instance \cite{HT1,K-L-S-S-2}. The well-known counterpart of Proposition~\ref{emb2} reads as
  \begin{equation}\label{comp-adm}
  \id^\beta: \Ae(\Rn, w^\beta)\hookrightarrow
\Az(\Rn)\quad \text{compact}\quad\iff\quad 
\frac{\beta}{p_1}>\frac{\nd}{p^*} \qquad\text{and}\qquad \delta > \frac{\nd}{p^*}.
  \end{equation}
\end{remark}

Now we turn our attention to Example~\ref{Ex-Muck}(ii) and the model weight $\wLog$. The compactness result reads as follows.
 
\begin{proposition}[{\cite[Prop.~3.9]{HaSk-lim}}]\label{emb3}
Let $\wLog$ be given by \eqref{wlog-1}, \eqref{wlog-2}. The embedding
\begin{equation}
\idlog : \be(\Rn, \wLog) \hookrightarrow  \bz(\Rn)
\label{dd-35}
\end{equation}
is compact if, and only if,  
\begin{align}
\label{dd-52}
&\begin{cases}
\text{either}&  \frac{\beta_1}{p_1} \ > \ \frac{\nd}{p^\ast},\quad  \beta_2\in\R, \\[1ex]
\text{or} &  \frac{\beta_1}{p_1} \ = \ \frac{\nd}{p^\ast}, \quad  
\frac{\beta_2}{p_1} > \frac{1}{p^\ast},\end{cases}\\
\intertext{and}
\label{dd-53}
&\begin{cases}
\text{either}&  \delta>\max\left(\frac{\alpha_1}{p_1},
  \frac{\nd}{p^\ast}\right),\quad   \alpha_2\in\R, \\[1ex]
\text{or} &  \delta=\frac{\alpha_1}{p_1} > \frac{\nd}{p^\ast}, \quad  \frac{\alpha_2}{p_1}>\frac{1}{q^\ast}.
\end{cases}
\end{align}
\end{proposition}

\begin{remark}\label{rem-pure-wlog}
  In case of $F$-spaces there is an almost complete characterisation in \cite[Cor.~3.15]{HaSk-lim}. For the `purely logarithmic' weight $w^{\log}_{\bm{\gamma}}$ given by \eqref{pure-log-w} the above result, cf. \cite[Prop.~3.9]{HaSk-lim} reads as follows:
  \[
\id_{\log}: \be(\rn, w^{\log}_{\bm{\gamma}}) \hookrightarrow \bz(\rn)\quad 
  \text{is compact}\quad\iff \quad p_1\leq p_2, \quad \delta>0,\quad \gamma_1\in\real, \quad \gamma_2>0. 
  \]
\end{remark}

%%%%%%%%%%%%%%%%%%%%%%%%5

\begin{remark}\label{rem-spaces-dom}
Let $\Omega\subset\rn$ be a bounded Lipschitz domain and $0<p,q\leq\infty$ (with $p<\infty$ in the
$F$-case), $s\in\real$. Let the spaces $\B(\Omega)$ and $\F(\Omega)$ be defined by restriction. It is well known that
\begin{equation}\label{id_Omega}
  \id_\Omega : \Ae(\Omega) \to \Az(\Omega)
\end{equation}
is compact, if, and only if,
\begin{equation}\label{id_Omega-comp}
s_1-s_2 > \nd\left(\frac{1}{p_1}-\frac{1}{p_2}\right)_+,
\end{equation}
where $s_i\in\real$, $0<p_i,q_i\leq\infty$ ($p_i<\infty$ if $A$=$F$), $i=1,2$.
\end{remark}

%%%%%%%%%%%%%%%%%%%5

\section{Nuclear embeddings}\label{sect-nuc}
Our main goal in this paper is to study nuclear embeddings between the weighted spaces introduced above. So we first recall some fundamentals of the concept and important results we rely on in the sequel.

\subsection{The concept and recent results}\label{subsec-conc-nuc}

Let $X,Y$ be Banach spaces, $T\in \mathcal{L}(X,Y)$ a linear and bounded operator. Then $T$ is called {\em nuclear}, denoted by $T\in\mathcal{N}(X,Y)$, if there exist elements $a_j\in X'$, the dual space of $X$, and $y_j\in Y$, $j\in\mathbb{N}$, such that $\sum_{j=1}^\infty \|a_j\|_{X'} \|y_j\|_Y < \infty$ and a nuclear representation $Tx=\sum_{j=1}^\infty a_j(x) y_j$ for any $x\in X$. Together with the {\em nuclear norm}
\[
 \nn{T}=\inf\Big\{ \sum_{j=1}^\infty   \|a_j\|_{X'} \|y_j\|_Y:\ T =\sum_{j=1}^\infty a_j(\cdot) y_j\Big\},
  \]
  where the infimum is taken over all nuclear representations of $T$, the space $\mathcal{N}(X,Y)$ becomes a Banach space. It is obvious that any nuclear operator can be approximated by finite rank operators, hence 
  nuclear operators are, in particular, compact.

  \begin{remark}
    This concept has been introduced by Grothendieck \cite{grothendieck} and was intensively studied afterwards, cf. 
\cite{Pie-snumb,Pie-op-2,pie-84} and also \cite{pie-history} for some history. At that time  applications were intended to better understand, for instance, nuclear locally convex spaces, operator ideals, eigenvalues of compact operators in Banach spaces.
There exist extensions of the concept to $r$-nuclear operators, $0<r<\infty$, where $r=1$ refers to the nuclearity. It is well-known that $\mathcal{N}(X,Y)$ possesses the ideal property. In Hilbert spaces $H_1,H_2$, the nuclear operators $\mathcal{N}(H_1,H_2)$ coincide with the trace class $S_1(H_1,H_2)$, consisting of those $T$ with singular numbers $(s_n(T))_n \in \ell_1$.
\end{remark}

  We collect some more or less well-known facts needed in the sequel.
\begin{proposition}\label{coll-nuc}
\begin{itemize}
\item[{\bfseries\upshape (i)}]  If $X$ is an $n$-dimensional Banach space, then 
$$\nn{\id:X\rightarrow X}= n. $$ 
\item[{\bfseries\upshape (ii)}]  For any Banach space $X$ and any bounded linear operator $T:\ell^n_\infty\rightarrow X$ we have 
\[\nn{T} = \sum_{i=1}^n \|Te_i\| .\]
\item[{\bfseries\upshape (iii)}]  If $T\in \mathcal{L}(X,Y)$ is a nuclear operator and $S\in \mathcal{L}(X_0,X)$ and $R\in \mathcal{L}(Y,Y_0)$, then $STR$ is a nuclear operator and 
\[ \nn{STR} \le \|S\| \|R\| \nn{T} . \] 

\end{itemize}
\end{proposition}
  
Already in the early years there was a strong interest to find further examples of nuclear operators beyond diagonal operators in $\ell_p$ spaces, where a complete answer was obtained in \cite{tong}. Let $\tau=(\tau_j)_{j\in\nat}$ be a scalar sequence and denote by $D_\tau$ the corresponding diagonal operator, $D_\tau: x=(x_j)_j \mapsto (\tau_j x_j)_j$, acting between $\ell_p$ spaces. Let us introduce the following notation: for numbers $r_1,r_2\in [1,\infty]$, let $\tn(r_1,r_2)$ be given by 
\begin{equation}\label{tongnumber}
\frac{1}{\tn(r_1,r_2)} = \begin{cases}
    1, & \text{if}\ 1\leq r_2\leq r_1\leq \infty, \\
    1-\frac{1}{r_1}+\frac{1}{r_2}, & \text{if}\ 1\leq r_1\leq r_2\leq \infty.
  \end{cases}
\end{equation}
Hence $1\leq \tn(r_1,r_2)\leq \infty$, and 
\[ \frac{1}{\tn(r_1,r_2)}= 1-\left(\frac{1}{r_1}-\frac{1}{r_2}\right)_+ \geq \frac{1}{r^\ast}= \left(\frac{1}{r_2}-\frac{1}{r_1}\right)_+\ ,\]
with $\tn(r_1,r_2)=r^\ast $ if, and only if, $\{r_1,r_2\}=\{1,\infty\}$.

Recall that $c_0$ denotes the subspace of $\ell_\infty$ containing the null sequences.

\begin{proposition}[{\cite[Thms.~4.3, 4.4]{tong}}]\label{prop-tong}
  Let $1\leq r_1,r_2\leq\infty$ and $D_\tau$ be the above diagonal operator.
\begin{itemize}
\item[{\bfseries\upshape (i)}] 
  Then $D_\tau$ is nuclear if, and only if, $\tau=(\tau_j)_j \in \ell_{\tn(r_1,r_2)}$, with $\ell_{\tn(r_1,r_2)}= c_0$ if $\tn(r_1,r_2)=\infty$. Moreover,
  \[
  \nn{D_\tau:\ell_{r_1}\to\ell_{r_2}} = \|\tau|{\ell_{\tn(r_1,r_2)}}\|.
  \]
\item[{\bfseries\upshape (ii)}]
  Let $n\in\nat$ and $D^n_\tau: \ell^n_{r_1}\to \ell^n_{r_2}$ be the corresponding diagonal operator $D_\tau^n: x=(x_j)_{j=1}^n \mapsto (\tau_j x_j)_{j=1}^n$. Then 
\begin{equation}\label{tong-res}
\nn{D_\tau^n:\ell_{r_1}^n\rightarrow \ell^n_{r_2}} = \left\| (\tau_j)_{j=1}^n | {\ell_{\tn(r_1,r_2)}^n} \right\|.
\end{equation}
\end{itemize}
\end{proposition}

\begin{example}
In the special case of $\tau\equiv 1$, i.e., $D_\tau=\id$, (i) is not applicable and (ii) reads as  
  \[
\nn{\id :\ell_{r_1}^n\rightarrow \ell^n_{r_2}} =
\begin{cases} 
n & \text{if}\qquad 1\le r_2\le r_1\le \infty,\\
n^{1-\frac{1}{r_1}+\frac{1}{r_2}} & \text{if}\qquad 1\le r_1\le r_2\le \infty .
\end{cases}
\]
In particular, $\nn{\id:\ell_1^n\rightarrow \ell^n_\infty}=1$. 
\end{example}

\begin{remark}
  The remarkable result (ii) can be found in \cite{tong}, see also \cite{Pie-op-2} for the case $p=1$, $q=\infty$. 
  \end{remark}

We return to the situation of compact embeddings of spaces on domains, as described in Remark~\ref{rem-spaces-dom}. Recently Triebel proved in \cite{Tri-nuclear} the following counterpart for its nuclearity.

\begin{proposition}[{\cite{Tri-nuclear}}]\label{prod-id_Omega-nuc}
  Let $\Omega\subset\rn$ be a bounded Lipschitz domain, $1<p_i,q_i<\infty$, $s_i\in\real$. Then the embedding $\id_\Omega$ given by \eqref{id_Omega} is nuclear if, and only if,
  \begin{equation}
    s_1-s_2 > \nd-\nd\left(\frac{1}{p_2}-\frac{1}{p_1}\right)_+.
\label{id_Omega-nuclear}
  \end{equation}
\end{proposition}

\begin{remark}
  The proposition is stated in \cite{Tri-nuclear} for the $B$-case only, but due to the independence of \eqref{id_Omega-nuclear} of the fine parameters $q_i$, $i=1,2$, and in view of (the corresponding counterpart of) \eqref{B-F-B} it can be extended immediately to $F$-spaces. The if-part of the above result is essentially covered by \cite{Pie-r-nuc} (with a forerunner in \cite{PiTri}). Also part of the necessity of \eqref{id_Omega-nuclear} for the nuclearity of $\id_\Omega$ was proved by Pietsch in \cite{Pie-r-nuc} such that only the limiting case $ s_1-s_2 = \nd-\nd(\frac{1}{p_2}-\frac{1}{p_1})_+$ was open for many decades. Only recently Edmunds, Gurka and Lang in \cite{EGL-3} (with a forerunner in \cite{EL-4}) obtained some answer in the limiting case which was then completely solved in \cite{Tri-nuclear}. In \cite{CoDoKu} the authors dealt with the nuclearity of the embedding $B^{s_1,\alpha_1}_{p_1,q_1}(\Omega)\to B^{s_2,\alpha_2}_{p_2,q_2}(\Omega)$ where the indices $\alpha_i$ represent some  additional logarithmic smoothness. They obtained a characterisation for almost all possible settings of the parameters.
  Note that in \cite{Pie-r-nuc} some endpoint cases (with $p_i,q_i\in \{1,\infty\}$) were already discussed for embeddings of Sobolev and certain Besov spaces (with $p=q$) into Lebesgue spaces. We are able to further extend Proposition~\ref{prod-id_Omega-nuc} in Corollary~\ref{prod-id_Omega-nuc-ext} below.
\end{remark}
  
\begin{remark}
In \cite{CoEdKu} some further limiting endpoint situations of nuclear embeddings like $\id:B^{\nd}_{p,q}(\Omega)\to L_p(\log L)_a(\Omega)$ are studied. For some weighted results see also \cite{Parfe-2}.
\end{remark}

%\open{possibly discuss endpoint cases $p_i,q_i\in \{1,\infty\}$ a bit more?}

\begin{remark}\label{rem-p*-tong}
  For later comparison we may reformulate the compactness and nuclearity characterisations of $\id_\Omega$ in \eqref{id_Omega-comp} and \eqref{id_Omega-nuclear} as follows, involving the number $\tn(p_1,p_2)$ defined in \eqref{tongnumber}. Let $1<p_i,q_i<\infty$, $s_i\in\real$. Then
  \begin{align*}
    \id_\Omega: \Ae(\Omega) \to \Az(\Omega) \quad  \text{is compact}\quad & \iff \quad \delta> \frac{\nd}{p^\ast}\qquad\text{and}\\
   \id_\Omega: \Ae(\Omega) \to \Az(\Omega) \quad \text{is nuclear}\quad & \iff \quad \delta> \frac{\nd}{\tn(p_1,p_2)}.
  \end{align*}
  Hence apart from the extremal cases $\{p_1,p_2\}=\{1,\infty\}$ (not admitted in Proposition~\ref{prod-id_Omega-nuc}) nuclearity is indeed stronger than compactness also in this setting, i.e.,
  \[
  \id_\Omega: \Ae(\Omega) \to \Az(\Omega) \quad  \text{is compact, but not nuclear}\quad \iff \quad  \frac{\nd}{p^\ast} < \delta \leq \frac{\nd}{\tn(p_1,p_2)}.
  \]
  We shall observe similar phenomena in the weighted setting later.
  \end{remark}

\subsection{Weighted spaces}\label{subsec-wnuc}

We begin with some general implication from Proposition~\ref{prod-id_Omega-nuc} for Muckenhoupt weights $w\in\mathcal{A}_\infty$. Here we benefit from the regularity result Proposition~\ref{meas-sing}, in particular, the observation recalled in Remark~\ref{rem-regular-w}.

\begin{corollary}\label{cor-w-nuc-nec}
  Let $1<p_i,q_i<\infty$, $s_i\in\real$, $w\in\mathcal{A}_\infty$. If the embedding
  \[
  \id_w: \Ae(\rn,w) \to \Az(\rn)
  \]
  is nuclear, then
  \begin{equation}\label{nucl-nec-w}
    s_1-s_2>\nd-\nd\left(\frac{1}{p_2}-\frac{1}{p_1}\right)_+,\qquad\text{i.e.,}\quad
\delta > \frac{\nd}{\tn(p_1,p_2)}.  
  \end{equation}
\end{corollary}

\bpr
Assume that $\id_w$ is nuclear and $\Omega$ is a regularity ball for $w$ which always exists according to Remark~\ref{rem-regular-w}. Consider now the spaces $\Ae(\Omega)$ and $\Az(\Omega)$ defined by restriction (and equipped with the equivalent norm induced by the regularity ball), together with the corresponding linear and bounded extension operator, cf. \cite{rychkov-2}. Then Proposition~\ref{coll-nuc}(iii) implies the nuclearity of $\id_\Omega:\Ae(\Omega)\to\Az(\Omega)$ which leads to \eqref{nucl-nec-w} by Proposition~\ref{prod-id_Omega-nuc}.
\epr

\begin{remark}\label{remark-nec-w-only}
  Later we can slightly extend the above result and incorporate limiting cases $p_i,q_i\in \{1,\infty\}$, see Corollary~\ref{cor-w-nuc-nec-ref} below. Note, that the above result is in general a necessary condition for nuclearity only, as the simple example $w\equiv 1\in\mathcal{A}_\infty$ shows: in that case the unweighted embedding $\id: \Ae(\rn)\to\Az(\Rn)$ is known to be never compact (let alone nuclear), no matter what the other parameters $s_i, p_i, q_i$ are.
\end{remark}
%\open{again some discussion on endpoint cases $p_i,q_i\in \{1,\infty\}$ ?}

We return to the weight function $\wab$ in Example~\ref{Ex-Muck}(i) and give the counterpart of Proposition~\ref{emb2}.

\begin{theorem}\label{wab-nuc}
  Let $\alpha>-\nd$, $\beta>-\nd$, $\wab$ be given by \eqref{wab}. Assume that {$1\le p_1< \infty$, $1\le p_2\le \infty$,  and $1\le q_i\le\infty$, $s_i\in\real$, $i=1,2$}. Then the embedding $\id_{\alpha,\beta}: \Ae(\Rn,\wab)\hookrightarrow
\Az(\Rn)$ is nuclear if, and only if, 
 \beq
\frac{\beta}{p_1}>\frac{\nd}{\tn(p_1,p_2)} \qquad\text{and}\qquad \delta >\max
\left(\frac{\nd}{\tn(p_1,p_2)},\frac{\alpha}{p_1}\right). 
\label{wab-n-1}
\eeq
\end{theorem}

\begin{remark}
  Note that dealing with the weighted setting we have the same phenomenon in Theorem~\ref{wab-nuc} compared with the compactness result Proposition~\ref{emb2}, as described in Remark~\ref{rem-p*-tong} for the situation of spaces on bounded domains: the stronger nuclearity condition \eqref{wab-n-1} is exactly achieved when $p^\ast$ is replaced by $\tn(p_1,p_2)$.  
\end{remark}

%%%%%%%%%%%%%%%%%%%%%%%%%%%%%%%%%%%%%%%%%%%%%
%%%%%%%%%%%%%%%%%%%%%%%%%%%%%%
%%%%%%%%%%%%%%%%%%%%%%%%%%%%%%

\bpr
First note that in view of \eqref{B-F-B} and the independence of \eqref{wab-n-1} from the fine parameters $q_i$, $i=1,2$, together with Proposition~\ref{coll-nuc}(iii), it is sufficient to consider the case $A=B$, i.e., the Besov spaces.\\

    {\em Step 1}.\quad We first deal with the sufficiency of \eqref{wab-n-1} for the nuclearity. We return to Remark~\ref{rem-wave-red} where we explained our general strategy. Thus, to show the nuclearity of $\id_{\alpha,\beta}$ it is equivalent to proving the nuclearity of
    \[
    \id : b^{\sigma_1}_{p_1,q_1}(\wab) \hookrightarrow b^{\sigma_2}_{p_2,q_2} \quad\text{with}\quad \sigma_i=s_i-\frac{\nd}{2}-\frac{\nd}{p_i}\ ,\quad  i=1,2,
    \]
    which is obviously equivalent to the nuclearity of
    \[
    \id: b^{\sigma_1-\sigma_2}_{p_1,q_1}(\wab) \hookrightarrow b^0_{p_2,q_2},
    \]
    which in view of $\sigma_1-\sigma_2=\delta$  can be written as
    \begin{equation}\label{nu-w-10}
    \id: b^\delta_{p_1,q_1} (\wab) \hookrightarrow \ell_{q_2}(\ell_{p_2}).
    \end{equation}
Note that
    \begin{equation}\label{nu-w-11}
\wab(Q_{j, m}) \ \sim \ 2^{-j \nd} \left\{\begin{array}{llr} 2^{-j\alpha}
 &\text{if} & m=0, \\[1ex] 
\left|2^{-j} m\right|^{\alpha} &\text{if} &
    1\leq |m| < 2^j, \\[1ex] \left|2^{-j} m\right|^{\beta} &\text{if} &
    |m| \geq 2^j. \end{array}\right.
    \end{equation}
    We define the projection
    \[
    \id_1:  b^\delta_{p_1,q_1} (\wab) \hookrightarrow \ell_{q_2}(\ell_{p_2}),\quad  \id_1:  (\lambda_{j,m})_{j,m} \mapsto (\tilde{\lambda}_{j,m})_{j,m}, \quad \tilde{\lambda}_{j,m}=\begin{cases} \lambda_{j,m}& \text{if}\ |m|<2^j,\\ 0, & \text{if}\ |m|\geq 2^j,\end{cases}
    \]
such that (in a slight abuse of notation) we can understand $\id_1$ as
    \[
    \id_1 : \ell_{q_1}\left(2^{j(\delta-\frac{\alpha}{p_1})} \ell_{p_1}^{2^{j \nd}}(|m|^\alpha)\right) \hookrightarrow \ell_{q_2}(\ell_{p_2}),
    \]
    with
    \[
    \left\| \lambda | \ell_{q_1}\left(2^{j(\delta-\frac{\alpha}{p_1})} \ell_{p_1}^{2^{j \nd}}(|m|^\alpha)\right)\right\|
= \Big\| \Big\{2^{j(\delta-\frac{\alpha}{p_1})}  \Big(\sum_{|m|<2^j} |\lambda_{j,m}|^{p_1}\ |m|^\alpha\Big)^{\frac{1}{p_1}}
\Big\}_{j\in\no} | \ell_{q_1}\Big\|.
        \]
We split
    \[
    \id :  b^\delta_{p_1,q_1} (\wab) \hookrightarrow \ell_{q_2}(\ell_{p_2}) \quad \text{into}\quad \id = \id_1+\id_2\quad\text{with}\quad \id_r: b^\delta_{p_1,q_1} (\wab) \hookrightarrow \ell_{q_2}(\ell_{p_2}), \ r=1,2.
    \]
    Now we study the nuclearity of $\id_1$ and $\id_2$. We further decompose $\id_1$ into
    \begin{equation}\label{nu-w-1}
    \id_1 = \sum_{j=0}^\infty \id_{1,j}\quad\text{with}\quad \id_{1,j}= Q_j \circ \id^j\circ P_j,
    \end{equation}
    where $P_j$ is the projection onto $\ell_{p_1}^{2^{j \nd}}(|m|^\alpha)$, hence
    \[
\left\|    P_j : \ell_{q_1}\left(2^{j(\delta-\frac{\alpha}{p_1})} \ell_{p_1}^{2^{j \nd}}(|m|^\alpha)\right) \rightarrow \ell_{p_1}^{2^{j \nd}}(|m|^\alpha)\right\| = 2^{-j(\delta-\frac{\alpha}{p_1})},
    \]
    $\id^j: \ell_{p_1}^{2^{j \nd}}(|m|^\alpha) \to \ell_{p_2}$ is the embedding on level $j$, and  $Q_j$ is the embedding of $\ell_{p_2}$ into $\ell_{q_2}(\ell_{p_2})$ with $\|Q_j: \ell_{p_2} \rightarrow \ell_{q_2}(\ell_{p_2})\|=1$. Thus Proposition~\ref{coll-nuc}(iii) yields
    \begin{equation}\label{nu-w-2}
    \nn{\id_{1,j}} \leq \nn{\id^j} 2^{-j(\delta-\frac{\alpha}{p_1})},\quad j\in\no.
    \end{equation}
    Consequently, \eqref{nu-w-1} and \eqref{nu-w-2} lead to
    \begin{equation}\label{nu-w-3}
\nn{\id_1} \leq \sum_{j=0}^\infty 2^{-j(\delta-\frac{\alpha}{p_1})} \nn{\id^j}.
    \end{equation}
    Next we decompose $\id^j$ into certain diagonal operators and the natural embedding,
    \[
    \id^j = \left(\id : \ell_{p_2}^{2^{j d}} \hookrightarrow \ell_{p_2}\right) \circ D_{-\alpha} \circ D_\alpha  
    \]
    with
    \begin{align*} 
      & D_\alpha : \ell_{p_1}^{2^{j \nd}}(|m|^\alpha) \to \ell_{p_1}^{2^{j \nd}}, \quad D_\alpha: \{\lambda_{j,m}\}_{|m|<2^j} \mapsto \{\lambda_{j,m} |m|^{\frac{\alpha}{p_1}}\}_{|m|<2^j},\qquad \left\|D_\alpha: \ell_{p_1}^{2^{j \nd}}(|m|^\alpha) \to \ell_{p_1}^{2^{j \nd}}\right\| = 1, \\
      & D_{-\alpha} : \ell_{p_1}^{2^{j \nd}} \to \ell_{p_2}^{2^{j \nd}}, \quad D_{-\alpha}: \{\mu_{j,m}\}_{|m|<2^j} \mapsto \{\mu_{j,m} |m|^{-\frac{\alpha}{p_1}}\}_{|m|<2^j},\qquad \nn{D_{-\alpha}} = \left\| \left\{|m|^{-\frac{\alpha}{p_1}}\right\}_{|m|<2^j}| {\ell^{2^{j\nd}}_{\tn(p_1,p_2)}}\right\|, \\
      & \id : \ell_{p_2}^{2^{j \nd}} \hookrightarrow \ell_{p_2}, \quad \id: \{\lambda_{j,m}\}_{|m|<2^j}\mapsto \{\tilde{\lambda}_{j,m}\}_{m\in\Zn}, \ \tilde{\lambda}_{j,m}=\begin{cases}\lambda_{j,m}, & |m|<2^j, \\ 0, & |m|\geq 2^j, \end{cases}\qquad
      \| \id : \ell_{p_2}^{2^{j \nd}} \hookrightarrow \ell_{p_2}\|=1,
\end{align*}
where we applied Proposition~\ref{prop-tong}, in particular \eqref{tong-res}.
    Thus
    \begin{equation}\label{nu-w-4}
    \nn{\id^j} \leq  \left\| \left\{|m|^{-\frac{\alpha}{p_1}}\right\}_{|m|<2^j}| {\ell^{2^{j\nd}}_{\tn(p_1,p_2)}}\right\|.
     \end{equation}
    It remains to calculate the latter norm. {First assume that  $\tn(p_1,p_2)<\infty$. In this case, }
    \begin{align}\nonumber
      \left\| \left\{|m|^{-\frac{\alpha}{p_1}}\right\}_{|m|<2^j}|{\ell^{2^{j\nd}}_{\tn(p_1,p_2)}}\right\|^{\tn(p_1,p_2)} &= \sum_{|m|<2^j} |m|^{-\frac{\alpha}{p_1} \tn(p_1,p_2)} \ = \ \sum_{k=0}^j \sum_{|m|\sim 2^k} |m|^{-\frac{\alpha}{p_1} \tn(p_1,p_2)} \\
      &\sim \sum_{k=0}^j 2^{-k\frac{\alpha}{p_1} \tn(p_1,p_2)} 2^{k\nd} \ = \ \sum_{k=0}^j 2^{k(\nd-\frac{\alpha}{p_1} \tn(p_1,p_2))}\nonumber \\
      & \sim \ \begin{cases} 2^{j(\nd-\frac{\alpha}{p_1} \tn(p_1,p_2))}, & \frac{\nd}{\tn(p_1,p_2)} > \frac{\alpha}{p_1}, \\
      j, &  \frac{\nd}{\tn(p_1,p_2)} = \frac{\alpha}{p_1}, \\ 1, & \frac{\nd}{\tn(p_1,p_2)} < \frac{\alpha}{p_1}.\end{cases}\label{nu-w-5}
      \end{align}
Thus \eqref{nu-w-3}, \eqref{nu-w-4} and \eqref{nu-w-5} result in
\begin{align*}
\nn{\id_1} \ & \leq \sum_{j=0}^\infty 2^{-j(\delta-\frac{\alpha}{p_1})} \begin{cases} 2^{j(\frac{\nd}{\tn(p_1,p_2)} -\frac{\alpha}{p_1} )}, & \frac{\nd}{\tn(p_1,p_2)} > \frac{\alpha}{p_1}, \\
  j^{\frac{1}{\tn(p_1,p_2)}}, &  \frac{\nd}{\tn(p_1,p_2)} = \frac{\alpha}{p_1}, \\ 1, & \frac{\nd}{\tn(p_1,p_2)} < \frac{\alpha}{p_1},\end{cases} \quad
 \sim \begin{cases} \ds \sum_{j=0}^\infty 2^{-j(\delta-\frac{\nd}{\tn(p_1,p_2)})}, & \frac{\nd}{\tn(p_1,p_2)} > \frac{\alpha}{p_1}, \\
\ds  \sum_{j=0}^\infty 2^{-j(\delta-\frac{\alpha}{p_1})} j^{\frac{1}{\tn(p_1,p_2)}}, &  \frac{\nd}{\tn(p_1,p_2)} = \frac{\alpha}{p_1}, \\ \ds \sum_{j=0}^\infty 2^{-j(\delta-\frac{\alpha}{p_1})}, & \frac{\nd}{\tn(p_1,p_2)} < \frac{\alpha}{p_1}.\end{cases} 
\end{align*}
Hence $\nn{\id_1}\leq c<\infty$ if $\delta> \max(\frac{\nd}{\tn(p_1,p_2)}, \frac{\alpha}{p_1})$ as assumed by \eqref{wab-n-1}.

{If $t(p_1,p_2)=\infty$, i.e., if $p_1=1$ and $p_2=\infty$, then  $\nn{\id^j} { \leq }1 $ if $\alpha\ge 0$  and $\nn{\id^j} { \leq } 2^{-j\alpha} $ if $\alpha < 0$. In consequence,  
\[
\nn{\id_1}\le \sum_{j=0}^\infty 2^{-j(\delta- \max(\alpha,0))}<\infty\qquad \text{if}\qquad \delta>\max(\alpha,0) .\]}

Next we deal with
    \[
    \id_2:  b^\delta_{p_1,q_1} (\wab) \hookrightarrow \ell_{q_2}(\ell_{p_2}),\quad  \id_2:  (\lambda_{j,m})_{j,m} \mapsto (\tilde{\lambda}_{j,m})_{j,m}, \quad \tilde{\lambda}_{j,m}=\begin{cases} \lambda_{j,m}& \text{if}\ |m|\geq 2^j,\\ 0 & \text{if}\ |m|< 2^j,\end{cases}
    \]
such that (in a slight abuse of notation) we can understand $\id_2$ as
    \[
    \id_2 : \ell_{q_1}\left(2^{j(\delta-\frac{\beta}{p_1})} \ell_{p_1}(|m|^\beta)\right) \hookrightarrow \ell_{q_2}(\ell_{p_2}),
    \]
    with
    \[
    \left\| \lambda | \ell_{q_1}\left(2^{j(\delta-\frac{\beta}{p_1})} \ell_{p_1}(|m|^\beta)\right)\right\|
= \Big\| \Big\{2^{j(\delta-\frac{\beta}{p_1})}  \Big(\sum_{|m|\geq 2^j} |\lambda_{j,m}|^{p_1}\ |m|^\beta\Big)^{\frac{1}{p_1}}
\Big\}_{j\in\no} | \ell_{q_1}\Big\|.
        \]
        Again we decompose
            \begin{equation}\label{nu-w-6}
    \id_2 = \sum_{j=0}^\infty \id_{2,j}\quad\text{with}\quad \id_{2,j}= \widetilde{Q}_j \circ \widetilde{\id}^j\circ \widetilde{P}_j,
    \end{equation}
    where $\widetilde{P}_j$ is the projection onto $\ell_{p_1}(|m|^\beta)$, hence
    \[
\left\|    \widetilde{P}_j : \ell_{q_1}\left(2^{j(\delta-\frac{\beta}{p_1})} \ell_{p_1}(|m|^\beta)\right) \rightarrow \ell_{p_1}(|m|^\beta)\right\| = 2^{-j(\delta-\frac{\beta}{p_1})},
    \]
    $\widetilde{\id}^j: \ell_{p_1}(|m|^\beta) \to \ell_{p_2}$ is the embedding on level $j$, and  $\widetilde{Q}_j$ is the embedding of $\ell_{p_2}$ into $\ell_{q_2}(\ell_{p_2})$ with $\|\widetilde{Q}_j: \ell_{p_2} \rightarrow \ell_{q_2}(\ell_{p_2})\|=1$. Proposition~\ref{coll-nuc}(iii) together with \eqref{nu-w-6} yield
    \begin{equation}\label{nu-w-7}
\nn{\id_2} \leq \sum_{j=0}^\infty 2^{-j(\delta-\frac{\beta}{p_1})} \nn{\widetilde{\id}^j}
    \end{equation}
    if $\widetilde{\id}^j$ is a nuclear map. So we proceed similar as above, 
    \[
    \widetilde{\id}^j = \left(\widetilde{\id} : \ell_{p_2}\hookrightarrow \ell_{p_2}\right) \circ D_{-\beta} \circ D_\beta  
    \]
    with
    \begin{align*} 
      & D_\beta : \ell_{p_1}(|m|^\beta) \to \ell_{p_1}, \quad D_\beta: \{\lambda_{j,m}\}_{|m|\geq 2^j} \mapsto \{\lambda_{j,m} |m|^{\frac{\beta}{p_1}}\}_{|m|\geq 2^j},\qquad \left\|D_\beta: \ell_{p_1}(|m|^\beta) \to \ell_{p_1}\right\| = 1, \\
      & D_{-\beta} : \ell_{p_1} \to \ell_{p_2}, \quad D_{-\beta}: \{\mu_{j,m}\}_{|m|\geq 2^j} \mapsto \{\mu_{j,m} |m|^{-\frac{\beta}{p_1}}\}_{|m|\geq 2^j},\qquad \nn{D_{-\beta}} = \left\| \left\{|m|^{-\frac{\beta}{p_1}}\right\}_{|m|\geq 2^j}|{\ell_{\tn(p_1,p_2)}}\right\|, \\
      & \widetilde{\id} : \ell_{p_2} \hookrightarrow \ell_{p_2}, \quad \widetilde{\id}: \{\lambda_{j,m}\}_{|m|\geq 2^j}\mapsto \{\tilde{\lambda}_{j,m}\}_{m\in\Zn}, \ \tilde{\lambda}_{j,m}=\begin{cases}\lambda_{j,m}, & |m| \geq 2^j, \\ 0, & |m|< 2^j, \end{cases}\qquad
      \| \widetilde{\id} : \ell_{p_2} \hookrightarrow \ell_{p_2}\|=1,
\end{align*}
where we applied Proposition~\ref{prop-tong}(i). Hence
    \begin{equation}\label{nu-w-8}
    \nn{\widetilde{\id}^j} \leq  \left\| \left\{|m|^{-\frac{\beta}{p_1}}\right\}_{|m|\geq 2^j}|{\ell_{\tn(p_1,p_2)}}\right\|.
     \end{equation}
    It remains to calculate that norm. If $t(p_1,p_2)<\infty$, then 
    \begin{align}\nonumber
      \left\| \left\{|m|^{-\frac{\beta}{p_1}}\right\}_{|m|\geq 2^j}|{\ell_{\tn(p_1,p_2)}}\right\|^{\tn(p_1,p_2)} &= \sum_{|m|\geq 2^j} |m|^{-\frac{\beta}{p_1} \tn(p_1,p_2)} \ = \ \sum_{k=j}^\infty \sum_{|m|\sim 2^k} |m|^{-\frac{\beta}{p_1} \tn(p_1,p_2)} \\
      &\sim \sum_{k=j}^\infty 2^{-k\frac{\beta}{p_1} \tn(p_1,p_2)} 2^{k\nd} \ = \ \sum_{k=j}^\infty 2^{k(\nd-\frac{\beta}{p_1} \tn(p_1,p_2))} \ \sim \ 2^{j(\nd-\frac{\beta}{p_1} \tn(p_1,p_2))}
\label{nu-w-9}
      \end{align}
using our assumption \eqref{wab-n-1}, i.e., $\frac{\nd}{\tn(p_1,p_2)} < \frac{\beta}{p_1}$.
Thus \eqref{nu-w-7}, \eqref{nu-w-8} and \eqref{nu-w-9} result in
\begin{align*}
  \nn{\id_2} \ & \leq \sum_{j=0}^\infty 2^{-j(\delta-\frac{\beta}{p_1})} 2^{j(\frac{\nd}{\tn(p_1,p_2)} -\frac{\beta}{p_1} )} \ = \ 
  \sum_{j=0}^\infty 2^{-j(\delta-\frac{\nd}{\tn(p_1,p_2)})} \ \leq\ c<\infty
\end{align*}
in view of (the second part of) \eqref{wab-n-1} again. 

{If $t(p_1,p_2)=\infty$, {that is, $p_1=1$ and $p_2=\infty$,} then $\widetilde{\id}^j$ is nuclear if {$\{|m|^{-\beta}\}_{|m|\geq 2^j} \in c_0\subset \ell_\infty$, recall Proposition~\ref{prop-tong}(i). This requires $\beta>0$ and leads to} 
  %  $\beta\ge 0$ and
  $\nn{\widetilde{\id}^j} \leq 2^{-j\beta}$. So 
\[ \nn{\id_2} \leq \sum_{j=0}^\infty 2^{-j\delta}<\infty \qquad \text{if}\qquad \delta>0.\]}
This concludes the argument for the sufficiency part.\\

{\em Step 2}.\quad Now we show the necessity of \eqref{wab-n-1} for the nuclearity of $\id_{\alpha,\beta}$ and begin with the global behaviour of the weight and have to prove that the nuclearity of $\id_{\alpha,\beta}$ implies $\frac{\beta}{p_1}>\frac{\nd}{\tn(p_1,p_2)}$. So assume $\frac{\beta}{p_1}\leq \frac{\nd}{\tn(p_1,p_2)}$.
We return to our above construction. Let $k\in\nat$ and  consider the following commutative diagram
$$
\begin{array}{ccc}\ell_{q_1}\left(2^{j(\delta-\frac{\beta}{p_1})} \ell_{p_1}(|m|^\beta)\right) & \xrightarrow{~\quad\id_2\quad~} &
\ell_{q_2}\left(\ell_{p_2}\right)
 \\ P_k \Big\uparrow & & \Big\downarrow Q_k\\
 \ell_{p_1}^{2^{k\nd}}(|m|^\beta) & \xrightarrow{~\quad\id^k\quad~} &\ell_{p_2}^{2^{k\nd}}
\end{array}
$$
where 
\[P_k: \{\mu_m\}_{|m|\leq 2^k} \mapsto \{\lambda_{j,m}\}_{j\in\no, |m|\geq 2^j}, \quad \lambda_{j,m}=\begin{cases} \mu_m, & j=0,\quad 1\leq |m|\leq 2^k, \\ 0, & \text{otherwise},\end{cases}
\]
and
\[
Q_k: \{\lambda_{j,m}\}_{j\in\no, |m|\geq 2^j} \ \mapsto \{\mu_m\}_{|m|\leq 2^k}, \quad \mu_m = \begin{cases} \lambda_{0,m}, & 1\leq |m|\leq 2^k, \\ 0, & \text{otherwise},
  \end{cases}
\]
such that $\|P_k\| = \|Q_k\|=1$, $k\in\nat$. Thus
\[
\nn{\id^k} \leq \nn{\id_2},\quad k\in\nat.
\]
Similar as above, let
 \begin{align*} 
      & {D}_\beta : \ell_{p_1}^{2^{k \nd}} \to \ell_{p_1}^{2^{k \nd}}(|m|^\beta), \quad {D}_\beta: \{\mu_m\}_{|m|\leq 2^k} \mapsto \{\mu_m |m|^{-\frac{\beta}{p_1}}\}_{|m|\leq 2^k},\qquad \left\|{D}_\beta: \ell_{p_1}^{2^{k \nd}} \to \ell_{p_1}^{2^{k \nd}}(|m|^\beta)\right\| = 1, \\
   & D_{-\beta} : \ell_{p_1}^{2^{k \nd}} \to \ell_{p_2}^{2^{k \nd}}, \quad D_{-\beta}: \{\mu_m\}_{|m|\leq 2^k} \mapsto \{\mu_m |m|^{-\frac{\beta}{p_1}}\}_{|m|\leq 2^k},\qquad \nn{D_{-\beta}} = \left\| \left\{|m|^{-\frac{\beta}{p_1}}\right\}_{|m|\leq 2^k}|{\ell^{2^{k\nd}}_{\tn(p_1,p_2)}}\right\|
   \end{align*}
 where we applied Proposition~\ref{prop-tong}, in particular \eqref{tong-res}. Then
 \begin{equation}\label{dd-4}
\left\| \left\{|m|^{-\frac{\beta}{p_1}}\right\}_{|m|\leq 2^k}|{\ell^{2^{k\nd}}_{\tn(p_1,p_2)}}\right\| = \nn{D_{-\beta}} = \nn{\id^k\circ {D}_\beta} \leq \|{D}_\beta\|\ \nn{\id^k} \leq \nn{\id_2}, \quad k\in\nat.
 \end{equation}
 On the other hand, parallel to \eqref{nu-w-5},
 \begin{align}\nonumber
      \left\| \left\{|m|^{-\frac{\beta}{p_1}}\right\}_{|m|\leq 2^k}|{\ell^{2^{k\nd}}_{\tn(p_1,p_2)}}|\right\|^{\tn(p_1,p_2)} &= \sum_{|m|\leq 2^k} |m|^{-\frac{\beta}{p_1} \tn(p_1,p_2)} \ = \ \sum_{l=0}^k \sum_{|m|\sim 2^l} |m|^{-\frac{\beta}{p_1} \tn(p_1,p_2)} \\
      &\sim \sum_{l=0}^k 2^{-l\frac{\beta}{p_1} \tn(p_1,p_2)} 2^{l\nd} \ = \ \sum_{l=0}^k 2^{l(\nd-\frac{\beta}{p_1} \tn(p_1,p_2))}\nonumber \\
      & \sim \ \begin{cases} 2^{k(\nd-\frac{\beta}{p_1} \tn(p_1,p_2))}, & \frac{\nd}{\tn(p_1,p_2)} > \frac{\beta}{p_1}, \\
      k, &  \frac{\nd}{\tn(p_1,p_2)} = \frac{\beta}{p_1},
      \end{cases}\label{nu-w-5c}
      \end{align}
for arbitrary $k\in\nat$. But this leads to a contradiction in \eqref{dd-4} for $\nn{\id_2}<\infty$ in the considered cases. Thus 
 $\frac{\beta}{p_1}> \frac{\nd}{\tn(p_1,p_2)}$. 

We are  left to deal with the local part of the weight which is related to the second condition in \eqref{wab-n-1}. {Since any nuclear map is compact, the nuclearity of $\id_{\alpha,\beta}$ implies its compactness which by Proposition~\ref{emb2} leads to $\delta>\frac{\alpha}{p_1}$. It remains to show 
  $\delta>\frac{\nd}{\tn(p_1,p_2)}$ in all admitted cases of the parameters.  If $1<p_i,q_i<\infty$, $i=1,2$, this is an immediate consequence of Corollary~\ref{cor-w-nuc-nec}.  If $\tn(p_1,p_2)=\infty$, i.e., $p_1=1$ and $p_2=\infty$,  then   ${\tn(p_1,p_2)} = {p^*}$ and the statement follows from Proposition~\ref{emb2} again. We are left to deal with the limiting cases of $p_i$ and $q_i$, $i=1,2$, in case of $\tn(p_1,p_2)<\infty$.}

{Assume that $\id$ is a  nuclear operator. 
  Then $\id_1$ is also a nuclear operator and $\nn{\id_1}\le \nn{\id}$. For a fixed $k\in \N$, let $\pi_k:\{1,\ldots, 2^{k\nd}\}\rightarrow \{m\in \Zn: 2^{k}\le |m|\le 2^{k+1}\} $ be a bijection. For simplicity we assume that  $\# \{m\in \Zn: 2^{k}\le |m|\le 2^{k+1}\} = 2^{k\nd}$, neglecting constants.

 First, let us consider the following commutative diagram 
\begin{equation}\label{diag}
\begin{array}{ccc}
\ell_{q_1}\left(2^{j(\delta-\frac{\alpha}{p_1})} \ell^{2^{j\nd}}_{p_1}(|m|^{\alpha})\right) & \xrightarrow{~\quad\id_1\quad~} &
\ell_{q_2}\left(\ell^{2^{j\nd}}_{p_2}\right)
 \\ \Pi_k \Big\uparrow & & \Big\downarrow Q_k\\
 \ell_{p_1}^{2^{k\nd}} & \xrightarrow{~\quad\id^k\quad~} &\ell_{p_2}^{2^{k\nd}}
\end{array}
\end{equation}
where 
\begin{equation}\label{Pk}
\Pi_k: \{\mu_i\}_{i=1,\ldots, 2^{k\nd}} \mapsto \{\lambda_{j,m}\}_{j\in\no, |m|\leq 2^j}, 
\quad \lambda_{j,m}=\begin{cases} 
\mu_i, & j=k+1,\quad m=\pi_k(i), \\ 
0, & \text{otherwise},\end{cases}
\end{equation}
  and
\[
Q_k: \{\lambda_{j,m}\}_{j\in\no, |m|\leq 2^j} \ \mapsto \{\mu_i\}_{i=1,\ldots,  2^{k\nd}}, \quad 
\mu_i = \lambda_{k+1,m} \quad \text{if}\quad i=\pi_k^{-1}(m).
\]
Both operators $Q_k$ and $\Pi_k$ are bounded. Moreover
\[\left\|Q_k: \ell_{q_2}\left(\ell^{2^{j\nd}}_{p_2}\right)\to  \ell_{p_2}^{2^{k\nd}}\right\|=1,  \quad k\in\nat,
\]
and
\[
\left\|\Pi_k: \ell_{q_1}\left(2^{j(\delta-\frac{\alpha}{p_1})} \ell^{2^{j\nd}}_{p_1}(|m|^{\alpha})\right)\to  \ell_{p_1}^{2^{k\nd}}\right\| =2^{(k+1)\delta},  \quad k\in\nat,
\]
since 
$|m|^\alpha\sim 2^{k\alpha}$ if $m\in \pi_k(\{1,\ldots 2^{k\nd}\})$. Thus
\begin{equation}\label{0709}
2^{\frac{k\nd}{\tn(p_1,p_2)}}=\nn{\id^k} \leq c 2^{k\delta}\nn{\id},\quad k\in\nat.
\end{equation}
{Hence, letting $k\to\infty$, we obtain} $\delta\ge \frac{\nd}{\tn(p_1,p_2)}$. 
 
It remains to exclude the case $\delta = \frac{\nd}{\tn(p_1,p_2)}>0$. 
The operator  $\id_1$ is nuclear, so there exist  $f_i\in (\ell_{q_1}(2^{j(\delta {-\frac{\alpha}{p_1}})}\ell_{p_1}^{2^{j\nd}}(|m|^\alpha)))'$
%= \ell_{q'_1}(2^{-j\delta}\ell_{p'_1})$ 
and $g_i\in \ell_{q_2}(\ell_{p_2}^{2^{j\nd}})$ such that 
\[
\id_1(\lambda)= \sum_{i=1}^\infty f_i(\lambda)g_i \qquad \text{with}\qquad \sum_{i=1}^\infty \left\|f_i | (\ell_{q_1}(2^{j(\delta {-\frac{\alpha}{p_1}})}\ell_{p_1}^{2^{j\nd}}(|m|^\alpha)))' \right\|\,\left\|g_i|\ell_{q_2}(\ell_{p_2}^{2^{j\nd}}) \right\|<\infty .  
\] 
We choose $0<\varepsilon<1$ and  take $i_o$ such that
\[
\sum_{i=i_o+1}^\infty \left\|f_i | (\ell_{q_1}(2^{j(\delta {-\frac{\alpha}{p_1}})}\ell_{p_1}^{2^{j\nd}}(|m|^\alpha)))' \right\|\,\left\|g_i|\ell_{q_2}(\ell_{p_2}^{2^{j\nd}}) \right\| <\varepsilon.
\]
Let $X_\varepsilon=\bigcap_{i=1}^{i_o} \ker f_i$ and $\id_\varepsilon(\lambda)= \sum_{i=i_o+1}^\infty f_i(\lambda)g_i $. The operator $\id_\varepsilon$ is nuclear and $\nn{\id_\varepsilon}<\varepsilon$. Moreover,  if $\lambda\in X_\varepsilon$, then $\lambda=\id_1(\lambda)=\id_\varepsilon(\lambda)$. We consider the subspaces  $X_k= \Pi_k(\ell_{p_1}^{2^{k\nd}} )$, $k\in \N$, cf. \eqref{Pk}. If $k$ is such that $2^{\nd k}>i_o$, then $\dim \left(  X_\varepsilon\cap X_k\right) \ge 2^{\nd k}-i_o$, since $\codim X_\varepsilon\le  i_o$. Now we can repeat the argument used in the diagram \eqref{diag} for the operator $\id_\varepsilon$. More precisely, 
if $k_\varepsilon = \dim  \left(X_\varepsilon\cap X_k\right) $, then we have the following commutative diagram 
\begin{equation}\label{diag2}
\begin{array}{ccc}
\ell_{q_1}\left(2^{j(\delta-\frac{\alpha}{p_1})} \ell^{2^{j\nd}}_{p_1}(|m|^{\alpha})\right) & \xrightarrow{~\quad{\id_1=\id_\varepsilon}\quad~} &
\ell_{q_2}\left(\ell^{2^{j\nd}}_{p_2}\right)
 \\ \Pi_{k,\varepsilon} \Big\uparrow & & \Big\downarrow Q_{k, \varepsilon}\\
 \ell_{p_1}^{k_\varepsilon} & \xrightarrow{~\quad {\id^{ k_\varepsilon}}\quad~} &\ell_{p_2}^{k_\varepsilon}
\end{array}
\end{equation}
with $\Pi_{k,\varepsilon} $ and $Q_{k,\varepsilon} $ defined in a similar way as above, i.e., $\Pi_{k,\varepsilon} $ is the restriction of $\Pi_{k} $ to $ \ell_{p_1}^{k_\varepsilon}$. Note that $\Pi_{k,\varepsilon}$ is a linear bijection of $\ell_{p_1}^{k_\varepsilon}$ onto $X_\varepsilon\cap X_k$, and $\id_1 = \id_\varepsilon$ on $X_k\ \cap X_\varepsilon$.  {Thus
\begin{align}
  \nn{\id^{k_\varepsilon}: \ell_{p_1}^{k_\varepsilon}\to\ell_{p_2}^{k_\varepsilon} } & = \nn{\Pi_{k,\varepsilon} \circ \id_\varepsilon\circ Q_{k,\varepsilon}}\nonumber \\
  & \leq \left\|\Pi_{k,\varepsilon}:  \ell_{p_1}^{k_\varepsilon}\to \ell_{q_1}\left(2^{j(\delta-\frac{\alpha}{p_1})} \ell^{2^{j\nd}}_{p_1}(|m|^{\alpha})\right)\right\| \nn{\id_\varepsilon} \left\|Q_{k,\varepsilon}: \ell_{q_2}\left(\ell^{2^{j\nd}}_{p_2}\right)\to \ell_{p_2}^{k_\varepsilon}\right\| \nonumber\\
  & \leq c\ %k_\varepsilon^\delta
     2^{k \delta}  \ \nn{\id_\varepsilon} < c\ 2^{k \delta}  \varepsilon. \label{1709}
\end{align}
On the other hand, in view of \eqref{tong-res},
\[
\nn{\id^{k_\varepsilon}: \ell_{p_1}^{k_\varepsilon}\to\ell_{p_2}^{k_\varepsilon} } = \dim \left(X_k\cap X_\varepsilon\right)^{\frac{1}{\tn(p_1,p_2)}} = k_\varepsilon^{\frac{1}{\tn(p_1,p_2)}} \geq \ \left(2^{\nd k}-i_o\right)^{\frac{1}{\tn(p_1,p_2)}}.
\]
Together with \eqref{1709} and in view of our assumption $\delta = \frac{\nd}{\tn(p_1,p_2)}$ we thus arrive at
\begin{equation}\label{1309}
  (2^{k\nd}-i_o)^{\frac{1}{\tn(p_1,p_2)}} < c'\ \varepsilon\ 2^{k\delta},\quad\text{that is,}\quad (1 -i_o2^{-k {\nd}})^{\frac{1}{\tn(p_1,p_2)}}< c'\ \varepsilon,\quad k\in\nat.
\end{equation}
}
Taking $k\rightarrow \infty$ with fixed $\varepsilon$ and $i_o$  we get the contradiction.
}
\epr
%%%%%%%%%%%%%%%%%%%%%%%
%%%%%%%%%%%%%%%%%%%%%%%%%%%%%%%%%%%
%%%%%%%%%%%%%%%%%%%%%%%%%%%%%%%%%

\begin{remark}    We briefly want to discuss the above result and compare it with the compactness criterion as recalled in Proposition~\ref{emb2}.
 %  \begin{minipage}{0.4\textwidth}
  In view of the parameters \eqref{param} we now naturally have to assume the Banach case situation, i.e., $p_i,q_i\geq 1$, $i=1,2$, when studying nuclearity. Moreover, as an easy observation shows, it might well happen that for certain parameter settings the compact embedding $\id_{\alpha,\beta}$ can never be nuclear, independent of the target space. This is, for instance, the case when $\frac{1}{p_1}+\frac{\beta}{p_1 \nd}<1$, as then \eqref{wab-n-1} for $\beta$ is never satisfied. Moreover, this excludes, in particular, an application of Theorem~\ref{wab-nuc} to the situation of Sobolev spaces, $\id_{\alpha,\beta}^W: W^{k_1}_{p_1}(\Rn,\wab)\hookrightarrow W^{k_2}_{p_2}(\Rn)$, $1< p_i< \infty$,  and $k_i\in\no$, $i=1,2$, 
  based on \eqref{W=F} and Theorem~\ref{wab-nuc} with $A=F$. Here we would need $\wab\in \mathcal{A}_{p_1}$ which, by Example~\ref{Ex-Muck}(i), reads as $-\nd<\alpha,\beta<\nd(p_1-1)$. But, as just observed, this contradicts \eqref{wab-n-1} for $\beta$. 
  So it very much depends on the source space, including the weight parameters, whether or not in some compactness case the embedding $\id_{\alpha,\beta}$ is even nuclear. \\

 % \end{minipage}\hfill\begin{minipage}{0.55\textwidth}
~ \hfill\begin{picture}(0,0)%
\includegraphics{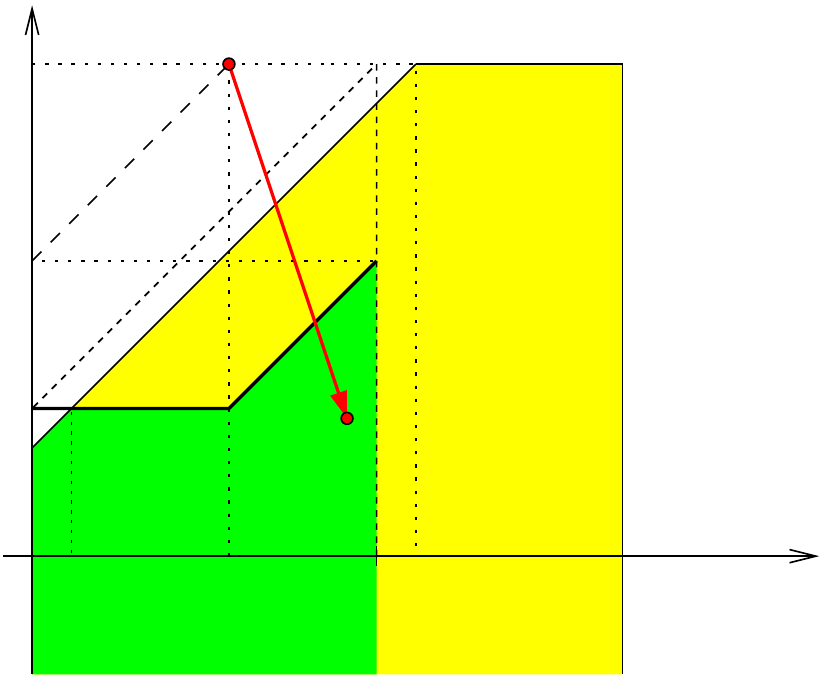}%
\end{picture}%
\setlength{\unitlength}{4144sp}%
\begingroup\makeatletter\ifx\SetFigFont\undefined%
\gdef\SetFigFont#1#2#3#4#5{%
  \reset@font\fontsize{#1}{#2pt}%
  \fontfamily{#3}\fontseries{#4}\fontshape{#5}%
  \selectfont}%
\fi\endgroup%
\begin{picture}(3759,3084)(304,-2413)
\put(1171,-2356){\makebox(0,0)[b]{\smash{{\SetFigFont{10}{12.0}{\familydefault}{\mddefault}{\updefault}{\color[rgb]{0,0,0}nuclear}%
}}}}
\put(361,-1231){\makebox(0,0)[rb]{\smash{{\SetFigFont{10}{12.0}{\familydefault}{\mddefault}{\updefault}{\color[rgb]{0,0,0}$s_1-\nd$}%
}}}}
\put(2656,-556){\makebox(0,0)[b]{\smash{{\SetFigFont{10}{12.0}{\familydefault}{\mddefault}{\updefault}{\color[rgb]{0,0,0}compact}%
}}}}
\put(1891,-61){\makebox(0,0)[lb]{\smash{{\SetFigFont{10}{12.0}{\familydefault}{\mddefault}{\updefault}{\color[rgb]{0,0,0}$\delta=\frac{\alpha}{p_1}$}%
}}}}
\put(361,-1456){\makebox(0,0)[rb]{\smash{{\SetFigFont{10}{12.0}{\familydefault}{\mddefault}{\updefault}{\color[rgb]{0,0,0}$s_1\!-\!\frac{\nd}{p_1}\!-\!\frac{\alpha}{p_1}$}%
}}}}
\put(3151,-2041){\makebox(0,0)[lb]{\smash{{\SetFigFont{10}{12.0}{\familydefault}{\mddefault}{\updefault}{\color[rgb]{0,0,0}$\frac{1}{p_1}\!+\!\frac{\beta}{\nd p_1}$}%
}}}}
\put(2116,-2041){\makebox(0,0)[lb]{\smash{{\SetFigFont{10}{12.0}{\familydefault}{\mddefault}{\updefault}{\color[rgb]{0,0,0}$\frac{1}{p_1}\!+\!\frac{\alpha}{\nd p_1}$}%
}}}}
\put(1216,-2041){\makebox(0,0)[rb]{\smash{{\SetFigFont{10}{12.0}{\familydefault}{\mddefault}{\updefault}{\color[rgb]{0,0,0}$\frac{1}{p_1}\!+\!\frac{\alpha}{\nd p_1}\!-\!1$}%
}}}}
\put(3961,-2041){\makebox(0,0)[b]{\smash{{\SetFigFont{10}{12.0}{\familydefault}{\mddefault}{\updefault}{\color[rgb]{0,0,0}$\frac1p$}%
}}}}
\put(1396,-2041){\makebox(0,0)[b]{\smash{{\SetFigFont{10}{12.0}{\familydefault}{\mddefault}{\updefault}{\color[rgb]{0,0,0}$\frac{1}{p_1}$}%
}}}}
\put(361,524){\makebox(0,0)[rb]{\smash{{\SetFigFont{10}{12.0}{\familydefault}{\mddefault}{\updefault}{\color[rgb]{0,0,0}$s$}%
}}}}
\put(1801,524){\makebox(0,0)[b]{\smash{{\SetFigFont{10}{12.0}{\familydefault}{\mddefault}{\updefault}{\color[rgb]{0,0,0}$\Ae(\Rn,\wab)$}%
}}}}
\put(2026,-2041){\makebox(0,0)[b]{\smash{{\SetFigFont{10}{12.0}{\familydefault}{\mddefault}{\updefault}{\color[rgb]{0,0,0}$1$}%
}}}}
\put(361,344){\makebox(0,0)[rb]{\smash{{\SetFigFont{10}{12.0}{\familydefault}{\mddefault}{\updefault}{\color[rgb]{0,0,0}$s_1$}%
}}}}
\put(991, 74){\makebox(0,0)[rb]{\smash{{\SetFigFont{10}{12.0}{\familydefault}{\mddefault}{\updefault}{\color[rgb]{0,0,0}$\delta=\frac{\nd}{p^\ast}$}%
}}}}
\put(361,-556){\makebox(0,0)[rb]{\smash{{\SetFigFont{10}{12.0}{\familydefault}{\mddefault}{\updefault}{\color[rgb]{0,0,0}$s_1\!-\!\frac{\nd}{p_1}$}%
}}}}
\put(1306,-1321){\makebox(0,0)[b]{\smash{{\SetFigFont{10}{12.0}{\familydefault}{\mddefault}{\updefault}{\color[rgb]{0,0,0}$\delta=\frac{\nd}{\tn[p_1,p_2)}$}%
}}}}
\put(2251,-1321){\makebox(0,0)[b]{\smash{{\SetFigFont{10}{12.0}{\familydefault}{\mddefault}{\updefault}{\color[rgb]{0,0,0}$\Az(\Rn)$}%
}}}}
\put(1689,-736){\makebox(0,0)[rb]{\smash{{\SetFigFont{10}{12.0}{\familydefault}{\mddefault}{\updefault}{\color[rgb]{1,0,0}$\id_{\alpha,\beta}$}%
}}}}
\end{picture}%
\hfill~\\

  %\end{minipage}
To illustrate the difference between compactness and nuclearity of $\id_{\alpha,\beta}$ in the area of parameters in the usual $(\frac1p,s)$ diagram above, where any space $\A$ is indicated by its smoothness and integrability (neglecting the fine index $q$), we have chosen the situation when
    \[\frac{\beta}{\nd p_1} >1 >\frac{\alpha}{\nd p_1} >1-\frac{1}{p_1}\geq 0.\]
\end{remark}

%%%%%%%%%%%%%%%%%%%%%%%%%%%%%%%%%%%%%%%%%%
%%%%%%%%%%%%%%%%%%%%%%%%%%%%%%%%%%%%%%%%%%%%%%%
    In the sense of Remark~\ref{rem-adm-comp} we can immediately conclude the nuclearity result for embeddings of spaces with admissible weights.

\begin{corollary}
 Let $\beta\geq 0$, $w^\beta(x)=\langle x\rangle^\beta$. Assume that {$1\le p_1< \infty$, $1\le p_2\le \infty$,  and $1\le q_i\le\infty$, $s_i\in\real$, $i=1,2$. } %$1<p_i,q_i<\infty$, $s_i\in\real$, $i=1,2$. 
 Then the embedding $\ \id^\beta: \Ae(\Rn,w^\beta)\hookrightarrow
\Az(\Rn)$ is nuclear if, and only if,  
\beq
\frac{\beta}{p_1}>\frac{\nd}{\tn(p_1,p_2)} \qquad\text{and}\qquad \delta > \frac{\nd}{\tn(p_1,p_2)}. 
\label{wab-n-1ad}
\eeq
    \end{corollary}
In view of \eqref{comp-adm}  we observe the phenomenon again that the nuclearity characterisation is distinct from the compactness one by replacing $p^\ast$ by $\tn(p_1,p_2)$ only. {In particular, when $\tn(p_1,p_2)=p^\ast$, that is, when $p_1=1$ and $p_2=\infty$ (recall that we always assume $p_1<\infty$), thus $A=F$, then nuclearity and compactness conditions coincide. {In that case $\tn(p_1,p_2)=p^\ast=\infty$ and Theorem~\ref{wab-nuc} together with Proposition~\ref{emb2} imply the following result.}}

{\begin{corollary}\label{cor-p1=1}
 Let $\alpha>-\nd$, $\beta>-\nd$, $\wab$ be given by \eqref{wab}.  Assume that  $1\le q_i\le\infty$, $s_i\in\real$, $i=1,2$.   
 The following conditions are equivalent 
 \bli
 \item[{\upshape\bfseries (i)}] the operator  $\ \id_{\alpha,\beta}: B_{1,q_1}^{s_1}(\Rn,\wab)\hookrightarrow  B_{\infty,q_2}^{s_2}(\Rn)$ is nuclear,
 \item[{\upshape\bfseries (ii)}] the operator  $\ \id_{\alpha,\beta}: B_{1,q_1}^{s_1}(\Rn,\wab)\hookrightarrow  B_{\infty,q_2}^{s_2}(\Rn)$ is compact,
 \item[{\upshape\bfseries (iii)}] $\beta>0$ and $\delta > \max (0,\alpha)$. 
\eli
    \end{corollary} } 

%%%%%%%%%%%%%%%%%%%%%%%%%%%%%%%%Another consequence of the theorem is a tiny extension of Proposition \ref{id_Omega-nuclear}  
{
We can also  extend Proposition~\ref{prod-id_Omega-nuc} to limiting cases $p_1$, $p_2$, $q_1$, $q_2$ equal to $1$ or $\infty$. The generalisation follows easily from Theorem~\ref{wab-nuc} for domains with the extension property, in particular for bounded Lipschitz domains. The sufficiency part has already been obtained in \cite[Thm.~4.2]{CoDoKu}, we may now complete the argument for the necessity part and thus partly extend \cite[Cor.~4.6]{CoDoKu}, i.e., when $\alpha_1=\alpha_2=0$ in the notation used in \cite{CoDoKu}.

\begin{corollary}\label{prod-id_Omega-nuc-ext}
  Let $\Omega\subset\rn$ be a bounded Lipschitz domain, % $1\le p_1<\infty$ and 
$s_i\in\real$,  $ 1\le p_i,q_i\le \infty$ ($p_i<\infty$ in the $F$-case) . Then 
  \begin{equation}\label{id_Omega-nuclear2}
   \id_\Omega: \Ae(\Omega) \to \Az(\Omega) \quad \text{is nuclear\quad if, and only if,}\qquad   \delta> \frac{\nd}{\tn(p_1,p_2)}.
\end{equation}  
\end{corollary} 

\begin{proof} Since the $q$-parameters play no role it is sufficient to prove the corollary for Besov spaces. The corresponding statement for the $F$-spaces follows then by elementary embeddings. 

{For the sufficiency part}
%\emph{Step 1}. Here
we benefit from the result \cite[Thm.~4.2]{CoDoKu} (with $\alpha_1=\alpha_2=0$ in their notation). 
The necessity can be proved in a way similar to the local  part in the Step 2 of the proof of Theorem~\ref{wab-nuc}.  Using the standard wavelet basis argument with Daubechies wavelets  we can factorise the embedding  $\ell_{q_1}(2^{j\delta} \ell^{2^{j\nd}}_{p_1}) \hookrightarrow \ell_{q_2}( \ell^{2^{j\nd}}_{p_2})$ through the embedding    $\id_\Omega: B^{s_1}_{p_1,q_1} (\Omega) \hookrightarrow B^{s_2}_{p_2,q_2}(\Omega)$. Then we can argue in the same way as in Step~2, \eqref{diag}-\eqref{1309} of   the proof of the last theorem.
\end{proof}}
%%%%%%%%%%%%%%%%%%%%%%%%%%%%%%%%%%%%%%%%

\begin{remark}
Parallel to Corollary~\ref{cor-p1=1} we can thus state that for arbitrary $q_1,q_2\in [1,\infty]$, \\[-2ex]
\begin{minipage}{0.5\textwidth}
\begin{align*}
  & A^{s_1}_{1,q_1}(\Omega) \hookrightarrow A^{s_2}_{\infty,q_2}(\Omega)\quad\text{compact}\\
   \iff\quad  &
 A^{s_1}_{1,q_1}(\Omega) \hookrightarrow A^{s_2}_{\infty,q_2}(\Omega)\quad\text{nuclear}\\
  \iff\quad & s_1-s_2>d,
\end{align*}
and
\begin{align*}
  & A^{s_1}_{\infty,q_1}(\Omega) \hookrightarrow A^{s_2}_{1,q_2}(\Omega)\quad\text{compact}\\
 \iff\quad & 
  A^{s_1}_{\infty,q_1}(\Omega) \hookrightarrow A^{s_2}_{1,q_2}(\Omega)\quad\text{nuclear}\\
 \iff\quad & s_1>s_2,
\end{align*}
recall Remark~\ref{rem-spaces-dom}. Hence in the extremal cases $\{p_1,p_2\}=\{1,\infty\}$ compactness and nuclearity coincide. In the usual $(\frac1p,s)$-diagram aside, where any space $\A(\Omega)$ is characterised by its parameters $s$ and $p$ (neglecting $q$), we indicated the parameter areas for $(\frac{1}{p_2}, s_2)$ (in dependence on a given original space $\Ae(\Omega)$ with $(\frac{1}{p_1},s_1)$) such that the corresponding embedding $\id_\Omega:\Ae(\Omega)\to\Az(\Omega)$ is compact or even nuclear.
\end{minipage}~~\hfill\begin{minipage}{0.43\textwidth}
\begin{picture}(0,0)%
\includegraphics{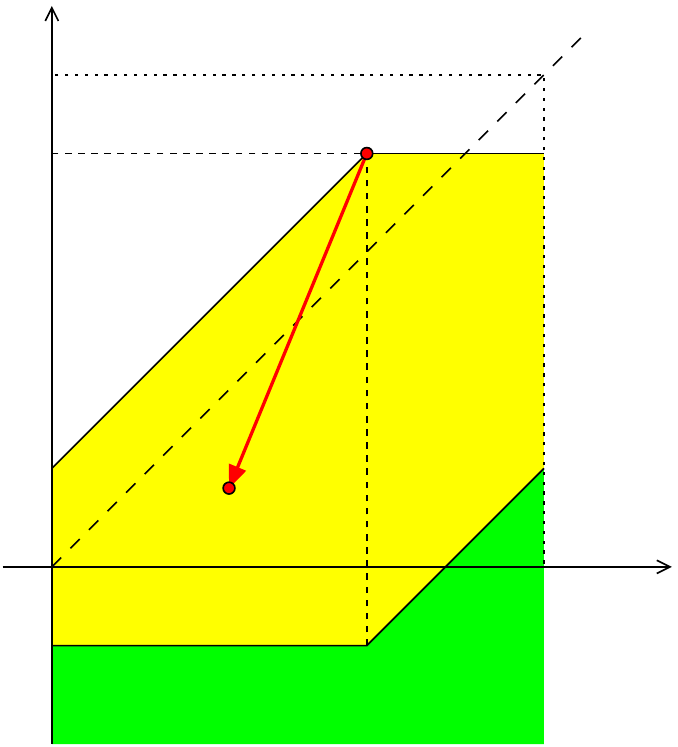}%
\end{picture}%
\setlength{\unitlength}{4144sp}%
\begingroup\makeatletter\ifx\SetFigFont\undefined%
\gdef\SetFigFont#1#2#3#4#5{%
  \reset@font\fontsize{#1}{#2pt}%
  \fontfamily{#3}\fontseries{#4}\fontshape{#5}%
  \selectfont}%
\fi\endgroup%
\begin{picture}(3084,3399)(1114,-4483)
\put(2161,-2266){\makebox(0,0)[rb]{\smash{{\SetFigFont{10}{12.0}{\familydefault}{\mddefault}{\updefault}{\color[rgb]{0,0,0}$\delta=\frac{\nd}{p^ \ast}$}%
}}}}
\put(1306,-1456){\makebox(0,0)[rb]{\smash{{\SetFigFont{10}{12.0}{\familydefault}{\mddefault}{\updefault}{\color[rgb]{0,0,0}$\nd$}%
}}}}
\put(1261,-4066){\makebox(0,0)[rb]{\smash{{\SetFigFont{10}{12.0}{\familydefault}{\mddefault}{\updefault}{\color[rgb]{0,0,0}$s_1-\nd$}%
}}}}
\put(1306,-3256){\makebox(0,0)[rb]{\smash{{\SetFigFont{10}{12.0}{\familydefault}{\mddefault}{\updefault}{\color[rgb]{0,0,0}$s_1-\frac{\nd}{p_1}$}%
}}}}
\put(1306,-1816){\makebox(0,0)[rb]{\smash{{\SetFigFont{10}{12.0}{\familydefault}{\mddefault}{\updefault}{\color[rgb]{0,0,0}$s_1$}%
}}}}
\put(1306,-1231){\makebox(0,0)[rb]{\smash{{\SetFigFont{10}{12.0}{\familydefault}{\mddefault}{\updefault}{\color[rgb]{0,0,0}$s$}%
}}}}
\put(3736,-1411){\makebox(0,0)[lb]{\smash{{\SetFigFont{10}{12.0}{\familydefault}{\mddefault}{\updefault}{\color[rgb]{0,0,0}$s=\frac{\nd}{p}$}%
}}}}
\put(2611,-1636){\makebox(0,0)[lb]{\smash{{\SetFigFont{10}{12.0}{\familydefault}{\mddefault}{\updefault}{\color[rgb]{0,0,0}$\Ae(\Omega)$}%
}}}}
\put(2791,-4156){\makebox(0,0)[b]{\smash{{\SetFigFont{10}{12.0}{\familydefault}{\mddefault}{\updefault}{\color[rgb]{0,0,0}$\delta=\frac{\nd}{\tn(p_1,p_2)}$}%
}}}}
\put(2791,-3841){\makebox(0,0)[rb]{\smash{{\SetFigFont{10}{12.0}{\familydefault}{\mddefault}{\updefault}{\color[rgb]{0,0,0}$\frac{1}{p_1}$}%
}}}}
\put(3601,-3841){\makebox(0,0)[lb]{\smash{{\SetFigFont{10}{12.0}{\familydefault}{\mddefault}{\updefault}{\color[rgb]{0,0,0}$1$}%
}}}}
\put(4096,-3841){\makebox(0,0)[b]{\smash{{\SetFigFont{10}{12.0}{\familydefault}{\mddefault}{\updefault}{\color[rgb]{0,0,0}$\frac1p$}%
}}}}
\put(2071,-3481){\makebox(0,0)[lb]{\smash{{\SetFigFont{10}{12.0}{\familydefault}{\mddefault}{\updefault}{\color[rgb]{0,0,0}$\Az(\Omega)$}%
}}}}
\put(3196,-2716){\makebox(0,0)[b]{\smash{{\SetFigFont{10}{12.0}{\familydefault}{\mddefault}{\updefault}{\color[rgb]{0,0,0}compact}%
}}}}
\put(2386,-2941){\makebox(0,0)[lb]{\smash{{\SetFigFont{10}{12.0}{\familydefault}{\mddefault}{\updefault}{\color[rgb]{1,0,0}$\id_\Omega$}%
}}}}
\put(1936,-4381){\makebox(0,0)[b]{\smash{{\SetFigFont{10}{12.0}{\familydefault}{\mddefault}{\updefault}{\color[rgb]{0,0,0}nuclear}%
}}}}
\end{picture}%

\end{minipage}
\end{remark}
\smallskip~

Corollary~\ref{prod-id_Omega-nuc-ext} leads immediately to an extended version of Corollary~\ref{cor-w-nuc-nec}.

\begin{corollary}\label{cor-w-nuc-nec-ref}
  Let $1\leq p_1<\infty$, $1\leq p_2,q_i\leq \infty$ ($p_i<\infty$ in the $F$-case), $s_i\in\real$, $i=1,2$, $w\in\mathcal{A}_\infty$. If the embedding
  \[
  \id_w: \Ae(\rn,w) \to \Az(\rn)
  \]
  is nuclear, then
  \begin{equation*}%\label{nucl-nec-w}
    s_1-s_2>\nd-\nd\left(\frac{1}{p_2}-\frac{1}{p_1}\right)_+,\qquad\text{i.e.,}\quad
\delta > \frac{\nd}{\tn(p_1,p_2)}.  
  \end{equation*}
\end{corollary}

\begin{proof}
One can copy the proof of Corollary~\ref{cor-w-nuc-nec} and benefit from the extension of  Proposition~\ref{prod-id_Omega-nuc} (used there) to the above Corollary~\ref{prod-id_Omega-nuc-ext}.
\end{proof}

\begin{remark}
  In the sense of Remark~\ref{remark-nec-w-only} we can add a further simple argument now, showing that the above criterion is a necessary one for nuclearity only: when $p_1=\infty$ and $w\in\mathcal{A}_\infty$ (arbitrary), then in view of \eqref{infty-w} the above embedding $\id_w$ is an unweighted one which is never compact (let alone nuclear).
\end{remark}

Now we study the counterpart of Theorem~\ref{wab-nuc} for the weight function $\wLog$ in Example~\ref{Ex-Muck}(ii). For convenience we recall the following well-known fact, which can also be found in \cite[Lemma~3.8]{HaSk-lim}.

\begin{lemma}\label{L-DD-1}
  Let $\gamma\in\R$, $\varkappa\in\R$, $j\in\N$. Then
  \[
  \sum_{k=1}^j 2^{k\gamma} k^\varkappa \ \sim \ \begin{cases} 2^{j\gamma} j^\varkappa, & \text{if}\quad \gamma>0, \\ 1, & \text{if}\quad \gamma<0,\end{cases}
%  \]
%  and, for $\gamma=0$,
%  \[
\qquad\text{and}\qquad
  \sum_{k=1}^j k^\varkappa \ \sim\ \begin{cases} 1,& \text{if}\quad \varkappa<-1,\\
    j^{1+\varkappa},& \text{if}\quad \varkappa>-1,\\ \log(1+j),& \text{if}\quad \varkappa=-1,\end{cases}
\]
always with equivalence constants independent of $j$.
\end{lemma}

Now we can give the counterpart of the compactness result Proposition~\ref{emb3}.

\begin{theorem}\label{wLog-nuc}
Let $\wLog$ be given by \eqref{wlog-1} with $\alpha_1>-\nd$, $\alpha_2\in\real$, $\beta_1>-\nd$, $\beta_2\in\real$. Assume that 
$1\le p_1,q_1<\infty$, 
$1\le p_2,q_2\le\infty$, $s_i\in\real$, $i=1,2$. Then the embedding $\id_{(\bm{\alpha},\bm{\beta})}: \be(\Rn,\wLog)\hookrightarrow
\bz(\Rn)$ is nuclear if, and only if,  
\begin{align}\label{wLog-n-1}
&\begin{cases}
\text{either}&  \frac{\beta_1}{p_1} \ > \ \frac{\nd}{\tn(p_1,p_2)},\quad  \beta_2\in\R, \\[1ex]
\text{or} &  \frac{\beta_1}{p_1} \ = \ \frac{\nd}{\tn(p_1,p_2)}, \quad  
\frac{\beta_2}{p_1} > \frac{1}{\tn(p_1,p_2)},\end{cases}\\
\intertext{and}
&\begin{cases}
\text{either}&  \delta>\max\left(\frac{\alpha_1}{p_1},
  \frac{\nd}{\tn(p_1,p_2)}\right),\quad   \alpha_2\in\R, \\[1ex]
\text{or} &  \delta=\frac{\alpha_1}{p_1} > \frac{\nd}{\tn(p_1,p_2)}, \quad  \frac{\alpha_2}{p_1}> \frac{1}{\tn(q_1,q_2)}.
\end{cases}\label{wLog-n-2}
\end{align}
\end{theorem}

\bpr %~ \\

{\em Step 1}. \quad We proceed essentially parallel to the arguments presented in the proof of Theorem~\ref{wab-nuc}. So again we may restrict ourselves to the study of the corresponding sequence spaces where the counterparts of \eqref{nu-w-10} and \eqref{nu-w-11} now read as 
 \begin{equation}\label{nu-w-10'}
    \id: b^\delta_{p_1,q_1} (\wLog) \hookrightarrow \ell_{q_2}(\ell_{p_2})
    \end{equation}
and
\begin{equation}\label{nu-w-11'}
\wLog(Q_{j, m}) \ \sim \ 2^{-j \nd} \left\{\begin{array}{llr} 2^{-j\alpha_1}
    (1+j)^{\alpha_2} &\text{if} & m=0, \\[1ex] 
\left|2^{-j} m\right|^{\alpha_1}
    \left(1 - \log \left|2^{-j} m\right|\right)^{\alpha_2} &\text{if} &
    1\leq |m| < 2^j, \\[1ex] \left|2^{-j} m\right|^{\beta_1}
    \left(1 + \log \left|2^{-j} m\right|\right)^{\beta_2} &\text{if} &
    |m| \geq 2^j. \end{array}\right.
\end{equation}
We split $\id=\id_1+\id_2$ as above, where only the weight $\wab$ has to be  replaced by $\wLog$. {First we consider the non-limiting case $  \delta>\max\left(\frac{\alpha_1}{p_1},
  \frac{\nd}{\tn(p_1,p_2)}\right)$.  By the same arguments as in the proof of Theorem~\ref{wab-nuc} we arrive at the following  counterpart of \eqref{nu-w-3},}
    \begin{equation}\label{nu-w-3'}
\nn{\id_1} \leq \sum_{j=0}^\infty 2^{-j(\delta-\frac{\alpha_1}{p_1})} \nn{\id^j}.
    \end{equation}
    The counterpart of \eqref{nu-w-4} is
    \begin{equation}\label{nu-w-4'}
    \nn{\id^j} \leq  \left\| \left\{|m|^{-\frac{\alpha_1}{p_1}}(1- \log \left|2^{-j} m\right|)^{-\frac{\alpha_2}{p_1}}\right\}_{|m|<2^j} | {\ell^{2^{j\nd}}_{\tn(p_1,p_2)}}\right\|.
     \end{equation}
    We calculate the  norm. First we assume that  $\tn(p_1,p_2)<\infty$. 
    Thus
    \begin{align}
\nonumber\lefteqn{      \left\| \left\{|m|^{-\frac{\alpha_1}{p_1}}(1- \log \left|2^{-j} m\right|)^{-\frac{\alpha_2}{p_1}}\right\}_{|m|<2^j}| {\ell^{2^{j\nd}}_{\tn(p_1,p_2)}}\right\|^{\tn(p_1,p_2)}}\\
\nonumber &= \sum_{|m|<2^j} |m|^{-\frac{\alpha_1}{p_1} \tn(p_1,p_2)}(1- \log \left|2^{-j} m\right|)^{-\frac{\alpha_2}{p_1} \tn(p_1,p_2)} \\
      \nonumber & = \ \sum_{k=0}^j \sum_{|m|\sim 2^k} |m|^{-\frac{\alpha_1}{p_1} \tn(p_1,p_2)}(1- \log \left|2^{-j} m\right|)^{-\frac{\alpha_2}{p_1} \tn(p_1,p_2)} \\
& \sim  \sum_{k=0}^j \ 2^{k\nd} 2^{-k\frac{\alpha_1}{p_1}\tn(p_1,p_2)}
    \left(1 + j-k\right)^{-\frac{\alpha_2}{p_1}
   \tn(p_1,p_2)} \notag\\
& \sim   \  2^{j(\nd-\frac{\alpha_1}{p_1}\tn(p_1,p_2))} \ \sum_{k=1}^{j+1} \
    2^{k(\frac{\alpha_1}{p_1}-\frac{\nd}{\tn(p_1,p_2)})\tn(p_1,p_2)} 
    k^{-\frac{\alpha_2}{p_1} \tn(p_1,p_2)},\label{dd-7}
    \end{align}
    such that \eqref{nu-w-4'} and Lemma~\ref{L-DD-1} imply
    \begin{align}\label{dd-25}
\nn{\id^j} \ \leq \
&   (1+j)^{-\frac{\alpha_2}{p_1}} \qquad \text{if}\quad \ds
  \frac{\alpha_1}{p_1} > \frac{\nd}{\tn(p_1,p_2)}, \  \alpha_2\in\R,\\[1ex] 
\label{dd-26}
\nn{\id^j} \ \leq \
& 2^{j(\frac{\nd}{\tn(p_1,p_2)}-\frac{\alpha_1}{p_1})}  \qquad
\text{if}\quad \frac{\alpha_1}{p_1} < \frac{\nd}{\tn(p_1,p_2)}, \ \alpha_2\in\R, 
\quad\text{or}\quad \ds \frac{\alpha_1}{p_1} = \frac{\nd}{\tn(p_1,p_2)},\
\frac{\alpha_2}{p_1}>\frac{1}{\tn(p_1,p_2)},\\[1ex]
\label{dd-27a}
\nn{\id^j} \ \leq \
& %2^{j(\frac{d}{\tn(p_1,p_2)}-\frac{\alpha_1}{p_1})}
(1+j)^{\frac{1}{\tn(p_1,p_2)}-\frac{\alpha_2}{p_1}}
\qquad \text{if}\quad \frac{\alpha_1}{p_1} = \frac{\nd}{\tn(p_1,p_2)}, \ 
  \frac{\alpha_2}{p_1}<\frac{1}{\tn(p_1,p_2)},\\[1ex]
%\label{dd-27b}
\nn{\id^j} \ \leq \
& %2^{j(\frac{d}{\tn(p_1,p_2)}-\frac{\alpha_1}{p_1})} 
  \log^{\frac{1}{\tn(p_1,p_2)}}(1+j) \qquad \text{if}\quad \frac{\alpha_1}{p_1} =
  \frac{\nd}{\tn(p_1,p_2)}, \  \frac{\alpha_2}{p_1}=\frac{1}{\tn(p_1,p_2)}. 
\label{nu-w-5'}
      \end{align}
We study the different cases to estimate $\nn{\id_1}$ by \eqref{nu-w-3'}. In case of \eqref{dd-25} we obtain that
\[
\nn{\id_1} \leq \sum_{j=0}^\infty 2^{-j (\delta-\frac{\alpha_1}{p_1})}
  (1+j)^{-\frac{\alpha_2}{p_1}} \leq c<\infty \qquad \text{if}\quad \ds
\frac{\alpha_1}{p_1} > \frac{\nd}{\tn(p_1,p_2)}, \quad {\delta>\frac{\alpha_1}{p_1}\quad \text{and}\quad \alpha_2\in\R.}
  \]
  In all other cases  \eqref{dd-26}--\eqref{nu-w-5'}, we obtain that
  \[
\nn{\id_1} \leq c<\infty \qquad \text{if}\quad  \delta>\frac{\nd}{\tn(p_1,p_2)}.
  \]
  Hence our assumption \eqref{wLog-n-2} ensures the nuclearity of $\id_1$.

Now let $\tn(p_1,p_2)=\infty$, i.e., $p_1=1$ and $p_2=\infty$. {Thus in a parallel way as above,}
 \begin{align}
\nonumber%\lefteqn{      
\left\| \left\{|m|^{-{{\alpha_1}}}(1- \log \left|2^{-j} m\right|)^{-{{\alpha_2}}}\right\}_{|m|<2^j}| {\ell^{2^{j\nd}}_{{\infty}}}\right\|
%\\ \nonumber &=
\le   \ C \  \begin{cases}
2^{-j\alpha_1} & \text{if}\; \alpha_1<0,\\
 (1+j)^{-\alpha_2} & \text{if}\;  \alpha_1=0\; \text{and}\; \alpha_2<0,\\
1 & \text{if}\; \alpha_1>0\; \text{or}\; \alpha_1=0\; \text{and}\; {\alpha_2 \geq 0}.
\end{cases}
%\sum_{|m|<2^j} |m|^{-\frac{\alpha_1}{p_1} \tn(p_1,p_2)}(1- \log \left|2^{-j} m\right|)^{-\frac{\alpha_2}{p_1}  \\ 
\end{align}
So 
  \begin{equation}\nonumber 
\nn{\id_1} \leq \sum_{j=0}^\infty 2^{-j(\delta-\alpha_1)} \nn{\id^j}< \infty \qquad \text{if}\quad  \delta> {\max(\alpha_1,0)}.
    \end{equation}

  We deal with $\id_2$ and again follow and adapt the arguments in the proof of Theorem~\ref{wab-nuc}. The counterparts of \eqref{nu-w-7} and \eqref{nu-w-8} lead to
  
  \begin{equation}\label{nu-w-7'}
\nn{\id_2} \leq \sum_{j=0}^\infty 2^{-j(\delta-\frac{\beta_1}{p_1})} \nn{\widetilde{\id}^j} \leq \ \sum_{j=0}^\infty 2^{-j(\delta-\frac{\beta_1}{p_1})}  \left\| \left\{|m|^{-\frac{\beta_1}{p_1}}\left(1 + \log \left|2^{-j} m\right|\right)^{-\frac{\beta_2}{p_1}}\right\}_{|m|\geq 2^j}| {\ell_{\tn(p_1,p_2)}}\right\|.
    \end{equation}
Now,  if $\tn(p_1,p_2)<\infty$, then
\begin{align*}
\left\| \left\{|m|^{-\frac{\beta_1}{p_1}}\left(1 + \log \left|2^{-j} m\right|\right)^{-\frac{\beta_2}{p_1}}\right\}_{|m|\geq 2^j}| {\ell_{\tn(p_1,p_2)}}\right\|^{\tn(p_1,p_2)}
 \ \sim & \sum_{|m|\geq 2^j} \left|m\right|^{-\frac{\beta_1}{p_1}\tn(p_1,p_2)}
    \left(1 + \log \left|2^{-j} m\right|\right)^{-\frac{\beta_2}{p_1}
    \tn(p_1,p_2)} \\
\sim & \sum_{l=j}^\infty \ \sum_{|m|\sim 2^l} \left|m\right|^{-\frac{\beta_1}{p_1}\tn(p_1,p_2)}
    \left(1 + \log \left|2^{-j} m\right|\right)^{-\frac{\beta_2}{p_1}\tn(p_1,p_2)} \\
\sim & \sum_{l=j}^\infty \ 2^{l(\nd-\frac{\beta_1}{p_1}\tn(p_1,p_2))}
    \left(1 + l-j\right)^{-\frac{\beta_2}{p_1}\tn(p_1,p_2)} 
\end{align*}
which by Lemma~\ref{L-DD-1} is finite if, and only if,
\begin{equation}
\label{dd-20}
\text{either}\quad \frac{\beta_1}{p_1}>\frac{\nd}{\tn(p_1,p_2)}, \ \beta_2\in \R,
\qquad\text{or}\quad \frac{\beta_1}{p_1}=\frac{\nd}{\tn(p_1,p_2)}, \
\frac{\beta_2}{p_1} > \frac{1}{\tn(p_1,p_2)},
\end{equation}
assumed by \eqref{wLog-n-1}.  If $\tn(p_1,p_2)=\infty$, {that is, $p_1=1$ and $p_2=\infty$, then for the nuclearity we first have to ensure that $\{|m|^{-\beta_1} (1+ \log \left|2^{-j} m\right|)^{-{\beta_2}}\}_{|m|\geq 2^j} \in c_0\subset \ell_\infty$, recall Proposition~\ref{prop-tong}(i). So we benefit from our assumption \eqref{wLog-n-1} which reads in this case as $\beta_1>0$ or $\beta_1=0$ and $\beta_2>0$. Furthermore, we conclude that}
 \begin{align}
\nonumber%\lefteqn{      
\left\| \left\{|m|^{-{\beta_1}}(1+ \log \left|2^{-j} m\right|)^{-{\beta_2}}\right\}_{|m|\ge 2^j}| {\ell_{{\infty}}}\right\|
%\\ \nonumber &=
\le   \ C \  \begin{cases}
2^{-j\beta_1} & \text{if}\; \beta_1>0,\\
1& \text{if}\;  \beta_1=0\; \text{and}\; \beta_2>0. 
\end{cases}
%\sum_{|m|<2^j} |m|^{-\frac{\alpha_1}{p_1} \tn(p_1,p_2)}(1- \log \left|2^{-j} m\right|)^{-\frac{\alpha_2}{p_1}  \\ 
\end{align}
{In other words, in both  cases we arrive at}
\begin{align*}
  & \left\| \left\{|m|^{-\frac{\beta_1}{p_1}}\left(1 + \log \left|2^{-j} m\right|\right)^{-\frac{\beta_2}{p_1}}\right\}_{|m|\geq 2^j}| {\ell_{\tn(p_1,p_2)}}\right\| \\
  & \sim  
{
\begin{cases}\ds
2^{j(\frac{\nd}{\tn(p_1,p_2)}-\frac{\beta_1}{p_1})} ,
  &\text{if}\ \frac{\beta_1}{p_1}>\frac{\nd}{\tn(p_1,p_2)}, \ \beta_2\in \R, \\
\ds 1,
  & \text{if}\ \frac{\beta_1}{p_1}=\frac{\nd}{\tn(p_1,p_2)}, \
\frac{\beta_2}{p_1} > \frac{1}{\tn(p_1,p_2)},
\end{cases}}    
\end{align*}
and \eqref{nu-w-7'} results in $\nn{\id_2}\leq c<\infty$ since $\delta>\frac{\nd}{\tn(p_1,p_2)}$ by \eqref{wLog-n-2}. This completes the proof of the sufficiency in the non-limiting case.\\

%%%%%%%%%%%%%%%
{\em Step 2}. Next we consider the limiting situation $\delta=\frac{\alpha_1}{p_1}>\frac{\nd}{\tn(p_1,p_2)}$ and $\frac{\alpha_2}{p_1}>\frac{1}{\tn(q_1,q_2)}$. We deal with the  case $\max\{\tn(p_1,p_2), \tn(q_1,q_2)\}<\infty$ and the  case $\max\{\tn(p_1,p_2),\tn(q_1,q_2)\}=\infty$  simultaneously. Now
 \[
    \id_1 : \ell_{q_1}\left( \ell_{p_1}^{2^{j \nd}}\big(|m|^{\alpha_1} (1-\log|2^{-j}m|)^{\alpha_2}\big)\right) \hookrightarrow \ell_{q_2}(\ell_{p_2}),
    \]
    with
    \[
    \left\| \lambda | \ell_{q_1}\left( \ell_{p_1}^{2^{j \nd}}\big(|m|^{\alpha_1} (1-\log|2^{-j}m|)^{\alpha_2}\big )\right)\right\|
= \Big\| \Big\{ \Big(\sum_{|m|<2^j} |\lambda_{j,m}|^{p_1}\ |m|^{\alpha_1} (1-\log|2^{-j}m|)^{\alpha_2}\Big)^{\frac{1}{p_1}}
\Big\}_{j\in\no} | \ell_{q_1}\Big\|.
        \]
Let $I_j=\{m\in \Z^{\nd}:\; 2^{j-1}\le |m|< 2^j\}$ if $j\in \N$ and $I_0=\{0\}$. 
We decompose $\id_1$ in the following way,
\begin{equation*}
\id_1 = \sum_{j=0}^\infty \widetilde{\id}_{1,j}  %= \sum_{j=0}^\infty \sum_{m\in I_j} \widetilde{\id}_{1,j,m} ,
\end{equation*}
where 
\begin{align*}
\left\{\widetilde{\id}_{1,j} \lambda\right\}_{k,m} =
\begin{cases}
\lambda_{k,m} & \text{if} \; k\ge j \;\text{and}\;m\in I_j,\\
0 & \text{otherwise}.
\end{cases}
\end{align*}

First we show that the operators $\widetilde{\id}_{1,j}$ are nuclear and that
\begin{equation*}
\nu(\widetilde{\id}_{1,j})\le c 2^{j(\frac{\nd}{\tn(p_1,p_2)}- \frac{\alpha_1}{p_1})} \left\|\left\{k^{-\frac{\alpha_2}{p_1}}\right\}_{k\geq j}|\ell_{\tn(q_1,q_2)}\right\|. 
\end{equation*}

In a similar way as above we factorise the operator $\widetilde{\id}_{1,j}$ through the diagonal operator. Now $j$ is fixed and  $m\in I_j$, i.e., $|m|\sim 2^{j}$. So we can take the  operator $D_{j}:\widetilde{\ell_{q_1}(\ell_{p_1}^{2^{j\nd}})}\rightarrow \widetilde{\ell_{q_2}(\ell_{p_1}^{2^{j\nd}})}$ defined on the mixed norm space 
\[ 
\widetilde{\ell_{q_1}(\ell_{p_1}^{2^{j\nd}})} =\left\{\lambda=\{\lambda_{\ell,m}\}_{\ell\in \N_0, \  m\in I_j} \ : \quad \|\lambda|\widetilde{\ell_{q_1}(\ell_{p_1}^{2^{j\nd}})}\| = \Big(\sum_{\ell=0}^\infty \Big(\sum_{m\in I_j} |\lambda_{\ell,m}|^{p_1}\Big)^{\frac{q_1}{p_1}}\Big)^{\frac{1}{q_1}}<\infty \right\}
\]
by 
\[
D_{j}: \{\lambda_{\ell,m}\}  \mapsto\{ \lambda_{\ell,m} |m|^{-\frac{\alpha_1}{p_1}} (\ell+1)^{-\frac{\alpha_2}{p_1}}\}, \qquad  \ell\in \N_0,\quad m\in I_j.
\] 
Similarly we define the target space $\widetilde{\ell_{q_2}(\ell_{p_2}^{2^{j\nd}})}$. 
Then
\begin{equation*}
\begin{CD}
\ell_{q_1}\left( \ell_{p_1}^{2^{j \nd}}\big(|m|^{\alpha_1} (1-\log|2^{-j}m|)^{\alpha_2}\big)\right) @>\widetilde{\id}_{1,j}>> \ell_{q_2}(\ell_{p_2}) \\
@VT_jVV @AAP_jA\\ 
\widetilde{\ell_{q_1}(\ell_{p_1}^{2^{j\nd}})} @>D_{j}>>\widetilde{\ell_{q_2}(\ell_{p_2}^{2^{j\nd}})}
\end{CD}
\end{equation*}
where 
\begin{align*}
 T_j: & \{\lambda_{k,m}\}  \mapsto \{ \lambda_{\ell,m} = \lambda_{k,m} |m|^{\frac{\alpha_1}{p_1}} (\ell+1)^{\frac{\alpha_2}{p_1}}\}_{\ell,m}, \quad   \text{if}\quad \ell=k-j\in \N_0 \quad \text{and}\quad m\in I_j, 
\intertext{and}
 P_j: & \{\lambda_{\ell,m}\}  \mapsto\{ \lambda_{j+\ell,m}\}, \qquad  \ell\in \N_0   \quad \text{and}\quad m\in I_j.
\end{align*}
Moreover  $\|P_j\|=1$ and  the norm $\|T_j\|$ is uniformly bounded in $j\in\no$.  

The operators 
\begin{align*}
D_{1}:\ell_{q_1}\rightarrow \ell_{q_2}, \qquad D_1: \{\gamma_{\ell}\}_{\ell\in\no}  \mapsto\{ (\ell+1)^{-\frac{\alpha_2}{p_1}}\gamma_{\ell} \}_{\ell\in\no}\\
\intertext{and}
D_{2}:\ell_{p_1}^{2^{j\nd}}\rightarrow \ell_{p_2}^{2^{j\nd}}, \qquad D_{2}: \{\mu_{m}\}_{m\in I_j}  \mapsto\{ |m|^{-\frac{\alpha_1}{p_1} }\mu_{m} \}_{m\in I_j} 
\end{align*}
are nuclear and 
\begin{equation}\label{D1D2n}
\nu(D_1) = \left\|\left\{(\ell+1)^{-\frac{\alpha_2}{p_1}}\right\}_{\ell\in\no} |\ell_{\tn(q_1,q_2)}\right\|, \qquad \nu(D_2) = \left\|\left\{|m|^{-\frac{\alpha_1}{p_1}}\right\}_{m\in I_j} |\ell^{2^{j\nd}}_{\tn(p_1,p_2)}\right\|.
\end{equation}
Let 
\begin{align*}
D_1(\gamma)=\sum_{k=0}^\infty a_k(\gamma) y_k,\qquad a_k\in (\ell_{q_1})'= \ell_{q'_1} \quad \text{and}\quad y_k\in \ell_{q_2}
\intertext{and}
D_2(\mu)=\sum_{\ell=0}^\infty b_\ell(\mu) x_\ell,\qquad b_\ell\in (\ell^{2^{j\nd}}_{p_1})'= \ell^{2^{j\nd}}_{p'_1} \quad \text{and}\quad x_\ell\in \ell^{2^{j\nd}}_{p_2}
\end{align*}
be the corresponding nuclear decompositions. We define the following (double) sequences,
\begin{align*}
  c_{k,\ell}=\{a_{k,i}b_{\ell, n}\}_{i\in\nat, n\in I_j},\qquad k,\ell\in\no, 
\intertext{and}
z_{k,\ell}= \{y_{k,i}z_{\ell, n}\}_{i\in\nat, n\in I_j},\qquad k,\ell\in\no.
\end{align*}
One can easily check that, for each $k,\ell\in\no$,
\begin{align*}
c_{k,\ell} \in \ell_{q'_1}\left(\ell_{p'_1}^{2^{\nd j}}\right)= \Big( \ell_{q_1}(\ell_{p_1}^{2^{j\nd}})\Big)' , \qquad
\left\|c_{k,\ell} |\ell_{q'_1}(\ell_{p'_1}^{2^{j\nd}}) \right\|  = \left\|a_{k} |\ell_{q'_1}\right\|  \left\|b_{\ell} |\ell_{p'_1}^{2^{j\nd}} \right\|, 
\intertext{and}
z_{k,\ell} \in \ell_{q_2}\left(\ell_{p_2}^{2^{j\nd}}\right), \qquad
\left\|z_{k,\ell} |\ell_{q_2}(\ell_{p_2}^{2^{j\nd}}) \right\|  = \left\|y_{k} |\ell_{q_2} \right\| \left\|x_{\ell} |\ell_{p_2}^{2^{j\nd}} \right\|.
\end{align*}
Moreover, 
\begin{equation}\label{D1D22}
\sum_{k,\ell} \left\|c_{k,\ell} |\ell_{q'_1}(\ell_{p'_1}^{2^{j\nd}})\right\| \left\|z_{k,\ell} |\ell_{q_2}(\ell_{p_2}^{2^{j\nd}}) \right\|  = 
\sum_k \left\|a_{k} |\ell_{q'_1} \right\| \left\|y_{k} |\ell_{q_2} \right\| \sum_\ell \left\|x_{\ell} |\ell_{p_2}^{2^{j\nd}} \right\| \left\|b_{\ell} |\ell_{p'_1}^{2^{j\nd}} \right\| <\infty.
\end{equation}
Direct calculations show that (appropriately interpreted) 
\[
D_j(\lambda) = D_1\left(\left\{ D_2\left(\left\{\lambda_{\ell,m}\right\}_{m\in I_j}\right)\right\}_{\ell\in\no}\right)
 = 
\sum_{k,\ell} c_{k,\ell}(\lambda) z_{k,\ell} .\]
So $D_j$ is a nuclear operator. Taking the infimum over all possible nuclear representations of $D_1$ and $D_2$ we get 
\begin{align*}
\nu(D_j)\le &\ \nu(D_1)\nu(D_2) \le  \left\|\left\{(\ell+1)^{-\frac{\alpha_2}{p_1}}\right\}_{\ell\in\no} |\ell_{\tn(q_1,q_2)}\right\| \  \left\|\left\{|m|^{-\frac{\alpha_1}{p_1}}\right\}_{m\in I_j} |\ell^{2^{j\nd}}_{\tn(p_1,p_2)}\right\|\\  \le &\  2^{j(\frac{\nd}{\tn(p_1,p_2)}-\frac{\alpha_1}{p_1})}\left\|\left\{(\ell+1)^{-\frac{\alpha_2}{p_1}}\right\}_{\ell\in\no} |\ell_{\tn(q_1,q_2)}\right\|,
 \nonumber
\end{align*} 
cf. \eqref{D1D2n} and \eqref{D1D22}. 
In consequence, 
\[\nu(\id_1)\le \sum_{j=0}^\infty  \nu(\widetilde{\id}_{1,j}) \le   \sum_{j=0}^\infty 2^{j(\frac{\nd}{\tn(p_1,p_2)}-\frac{\alpha_1}{p_1})}\left\|\left\{k^{-\frac{\alpha_2}{p_1}}\right\}_{k\geq j} | {\ell_{\tn(q_1,q_2)}} \right\| <\infty
 \] 
since $\frac{\nd}{\tn(p_1,p_2)}<\frac{\alpha_1}{p_1}$ and $\frac{\alpha_2}{p_1}>\frac{1}{\tn(q_1,q_2)}$.  This completes the proof of the sufficiency.

%%%%%%%%%%%%%%%%%%%%%%%%%%%%%
{\em Step 3}.\quad It remains to show the necessity of \eqref{wLog-n-1}, \eqref{wLog-n-2} when $\id_{(\bm{\alpha},\bm{\beta})}: \be(\Rn,\wLog)\hookrightarrow \bz(\Rn)$ is nuclear. First we collect what is immediately clear by Corollary~\ref{cor-w-nuc-nec} and Proposition~\ref{emb3}, in the same spirit as in the beginning of Step~2 of the proof of Theorem~\ref{wab-nuc}. Thus the nuclearity of $\id_{(\bm{\alpha},\bm{\beta})}$ implies
\[
\delta>\frac{\nd}{\tn(p_1,p_2)}, \quad \text{and}\quad\delta>\frac{\alpha_1}{p_1}\quad\text{or}\quad \delta=\frac{\alpha_1}{p_1}\quad \text{and}\ \frac{\alpha_2}{p_1} > \frac{1}{q^\ast}.
\]
Moreover, in the limiting cases $\tn(p_1,p_2)=\infty$ or   $\tn(q_1,q_2)=\infty$ the sufficient conditions coincide with the conditions for compactness, therefore they are necessary.  

Let  $\tn(p_1,p_2)<\infty$  and   $\tn(q_1,q_2)<\infty$. Using \cite[Cor.~3.11]{HaSk-lim} we get
$\be(\rn, w_{\tilde{\alpha},\tilde{\beta}}) \hookrightarrow \be(\rn,\wLog)$ where $\tilde{\alpha}<\alpha_1$ or $\tilde{\alpha}=\alpha_1$ and $\alpha_2\leq 0$, and $\tilde{\beta}>\beta_1$, or $\tilde{\beta}=\beta_1$ and $\beta_2\leq 0$. Thus the nuclearity of $\id_{(\bm{\alpha},\bm{\beta})}$ implies the nuclearity of $\id_{\tilde{\alpha},\tilde{\beta}}$ which by Theorem~\ref{wab-nuc} leads, in particular, to
\[
\frac{\tilde{\beta}}{p_1}> \frac{\nd}{\tn(p_1,p_2)},
\]
hence $\frac{\beta_1}{p_1}\geq \frac{\nd}{\tn(p_1,p_2)}$, $\beta_2\in\real$, or $\frac{\beta_1}{p_1}> \frac{\nd}{\tn(p_1,p_2)}$ and $\beta_2\leq 0$. So we are left to deal with the limiting cases in \eqref{wLog-n-1}, \eqref{wLog-n-2}, that is, when
\[
\frac{\beta_1}{p_1} = \frac{\nd}{\tn(p_1,p_2)}\qquad\text{and}\qquad \delta=\frac{\alpha_1}{p_1}>\frac{\nd}{\tn(p_1,p_2)}.
\]
We prove it by contradiction and assume first $\frac{\beta_1}{p_1} = \frac{\nd}{\tn(p_1,p_2)}$, but $\frac{\beta_2}{p_1}\leq \frac{1}{\tn(p_1,p_2)}$. We follow essentially the same argument as in Step 2 of the proof of Theorem~\ref{wab-nuc}. The counterpart of \eqref{dd-4} reads now as
 \begin{equation}\label{dd-5}
   \left\| \left\{|m|^{-\frac{\beta_1}{p_1}}(1- \log \left|2^{-k} m\right|)^{-\frac{\beta_2}{p_1}}\right\}_{|m|\leq 2^k} | {\ell^{2^{k\nd}}_{\tn(p_1,p_2)}}\right\|
   %= \nn{D_{-\beta}} = \nn{\id^k\circ \widetilde{D}_\beta} \leq \|\widetilde{D}_\beta\|\ \nn{\id^k}
   \leq \nn{\id_2}, \quad k\in\nat.
 \end{equation}
On the other hand, similar to \eqref{dd-7},
 \begin{align*}\nonumber
      \left\| \left\{|m|^{-\frac{\beta_1}{p_1}}(1- \log \left|2^{-k} m\right|)^{-\frac{\beta_2}{p_1}}\right\}_{|m|\leq 2^k}| {\ell^{2^{k\nd}}_{\tn(p_1,p_2)}}\right\|^{\tn(p_1,p_2)} &= \sum_{|m|\leq 2^k} |m|^{-\frac{\beta_1}{p_1} \tn(p_1,p_2)}(1- \log \left|2^{-k} m\right|)^{-\frac{\beta_2}{p_1} \tn(p_1,p_2)} \\
      \nonumber & = \ \sum_{l=0}^k \sum_{|m|\sim 2^l} |m|^{-\frac{\beta_1}{p_1} \tn(p_1,p_2)}(1- \log \left|2^{-k} m\right|)^{-\frac{\beta_2}{p_1} \tn(p_1,p_2)} \\
& \sim  \sum_{l=0}^k \ 2^{l\nd} 2^{-l\frac{\beta_1}{p_1}\tn(p_1,p_2)}
    \left(1 + k-l\right)^{-\frac{\beta_2}{p_1}
   \tn(p_1,p_2)} \notag\\
& \sim   \ \sum_{l=1}^{k+1} \
    l^{-\frac{\beta_2}{p_1} \tn(p_1,p_2)},
    \end{align*}
 such that
\[
\left\| \left\{|m|^{-\frac{\beta_1}{p_1}}(1- \log \left|2^{-k} m\right|)^{-\frac{\beta_2}{p_1}}\right\}_{|m|\leq 2^k} | {\ell^{2^{k\nd}}_{\tn(p_1,p_2)}} \right\|\sim
\begin{cases}
 (1+k)^{\frac{1}{\tn(p_1,p_2)}-\frac{\beta_2}{p_1}}, & \text{if}\quad  \frac{\beta_2}{p_1}<\frac{1}{\tn(p_1,p_2)},\\[1ex]
  \log^{\frac{1}{\tn(p_1,p_2)}}(1+k), & \text{if}\quad \frac{\beta_2}{p_1}=\frac{1}{\tn(p_1,p_2)}\end{cases} 
\]
which again leads to a contradiction in \eqref{dd-5} if $k\to\infty$ since $\nn{\id_2}<\infty$.

We finally deal with the case $\delta=\frac{\alpha_1}{p_1}>\frac{\nd}{\tn(p_1,p_2)}$, $\frac{1}{q^\ast}< \frac{\alpha_2}{p_1}\leq \frac{ 1}{\tn(q_1,q_2)}$. Consider the commutative diagram
$$
\begin{array}{ccc}\ell_{q_1}\left( \ell_{p_1}^{2^{j \nd}}\big(|m|^{\alpha_1} (1-\log|2^{-j}m|)^{\alpha_2}\big)\right) & \xrightarrow{~\quad\id_1\quad~} &
\ell_{q_2}\left(\ell_{p_2}\right)
 \\ P_\ell \Big\uparrow & & \Big\downarrow Q_\ell\\
 \ell_{q_1}^{2^{\ell}}((1+k)^{\alpha_2}) & \xrightarrow{~\quad\id^\ell\quad~} &\ell_{q_2}^{2^{\ell}}
\end{array}
$$
where 
\[P_\ell: \{\mu_k\}_{0\le k < 2^{\ell}} \mapsto \{\lambda_{j,m}\}_{j\in\no, |m|< 2^j}, \quad 
\lambda_{j,m}=
\begin{cases} 
\mu_k, & j=k,\quad m=0, \\
 0, & \text{otherwise},
 \end{cases}
\]
and
\[
Q_\ell: \{\lambda_{j,m}\}_{j\in\no, |m|< 2^j} \ \mapsto \{\mu_k\}_{0\le k<2^\ell}, \quad 
\mu_k = \begin{cases} 
\lambda_{j,0}, & k=j, \\ 0, & \text{otherwise},
  \end{cases}
\]
such that $\|P_\ell\| = \|Q_\ell\|=1$, $k\in\no$. Thus
\[
\nn{\id_1}\ge \nn{\id^\ell} = \left\| \left\{(1+k)^{-\frac{\alpha_2}{p_1}}\right\}_{k<2^\ell} |\ell_{\tn(q_1,q_2)}^{2^\ell} \right\| .
\]
But $\| \{(1+k)^{-\frac{\alpha_2}{p_1}}\}_{k<2^\ell} |\ell_{\tn(q_1,q_2)}^{2^\ell} \|\rightarrow \infty$ when  $\ell\rightarrow \infty$ if $\frac{\alpha_2}{p_1}\leq \frac{ 1}{\tn(q_1,q_2)}$. This again leads to a contradiction since $\nn{\id_1}<\infty$.
%\open{missing}
\epr

\begin{remark}\label{remlim}
	If $ \delta>\max(\frac{\alpha_1}{p_1}, \frac{\nd}{\tn(p_1,p_2)})$  and  $\alpha_2\in\R$, then the condition \eqref{wLog-n-1} implies the nuclearity of the embedding 
	$\id_{(\bm{\alpha},\bm{\beta})}: \be(\Rn,\wLog)\hookrightarrow
	\bz(\Rn)$ for $1\le p_1<\infty$, $1\le p_2\le \infty$ and $1\le q_1,q_2 \le \infty$. This can be easily seen rewriting the sufficiency part of the above proof literally. Moreover, by elementary embeddings this statement holds also for 	$\id_{(\bm{\alpha},\bm{\beta})}: \fe(\Rn,\wLog)\hookrightarrow
	\fz(\Rn)$. 
	%,\quad   ,
\end{remark}

The next statement is a direct consequence of Proposition~\ref{emb3}, in particular, Remark~\ref{rem-pure-wlog},  Theorem~\ref{wLog-nuc} and Remark~\ref{remlim}.

\begin{corollary} \label{cor-pure-log}
  Let $\bm{\gamma}=(\gamma_1,\gamma_2)\in\real^2$ and $w^{\log}_{\bm{\gamma}}$ be given by \eqref{pure-log-w}. Assume that $1\le p_1<\infty$, $1\le p_2\le \infty$, $1\le q_i\le \infty$, $s_i\in\real$, $i=1,2$. Then 
  \[
\id_{\log}: \Ae(\rn, w^{\log}_{\bm{\gamma}}) \hookrightarrow \Az(\rn)
  \]
is compact if, and only if,  $\delta>0$, $p_1\leq p_2$ and $\gamma_2>0$. 

The embedding is nuclear if, and only if,  $\delta>0$, $\gamma_2>0$ , $p_1=1$ and $p_2=\infty$. 
\end{corollary}

\begin{proof}
Recall our Remark~\ref{rem-pure-wlog} for the compactness. As for nuclearity, we apply Theorem~\ref{wLog-nuc} with $\alpha_1=\beta_1=0$ and observe, that \eqref{wLog-n-1} is never satisfied unless $\tn(p_1,p_2)=\infty$.
\end{proof}

\subsection{Radial spaces}\label{subsec-rad}
%\open{missing}
So far we considered embeddings within the scale of spaces $\A$ which are compact -- and studied the question whether they are even nuclear. In case of spaces on bounded domains $\Omega$ or weighted spaces on $\rn$ it is well-known that compactness can appear, unlike in case of unweighted spaces on $\rn$. Furthermore, such Sobolev-type embeddings can also be compact in presence of symmetries, i.e., if we restrict our attention to subspaces consisting of distributions that satisfy certain symmetry conditions, in particular, if they are  radial. We want to consider this setting now. Here the sufficient and necessary conditions for the nuclearity of the  compact embeddings can be  easily proved due to the relation between subspaces of radial distributions and appropriately weighted spaces. Indeed the  conditions  follow from Theorem~\ref{wab-nuc}. We start with  recalling the  definition of  radial subspaces of Besov and Triebel-Lizorkin spaces. 

%%%%%%%%%%%%%%
Let $\Phi$ be an isometry of $\Rn$.  For $g \in {\cal S}(\Rn)$
we put $g^\Phi(x)=g(\Phi x)$. If $f\in {\cal S}'(\Rn)$, then
$f^\Phi$ is a tempered distribution defined by
\[ 
f^\Phi(g)=f(g^{\Phi^{-1}})\;, \quad g\in {\cal S}(\Rn),
\]
where $\Phi^{-1}$ denotes the isometry inverse to $\Phi$. 
 
\begin{definition}\label{invader}
Let $SO(\Rn )$ be the group of rotations around the origin in $\Rn$.
We say that the tempered distribution $f$ is invariant with
respect to $SO(\Rn )$ if $f^\Phi=f$ for any $\Phi\in SO(\Rn )$. For any possible
$s,p,q$ we put 
\begin{eqnarray*}
RA^s_{p,q}(\Rn)&=&\{f\in A^s_{p,q}(\Rn): f\;\mbox{\rm is
invariant with respect to}\;SO(\Rn )\}\, .
\end{eqnarray*}
\end{definition}

\noindent
\begin{remark}
The space $RA^s_{p,q}(\Rn)$ is a
closed subspace of $A^s_{p,q}(\Rn)$. 
Thus, it is a Banach space  with respect to the induced norm if $p,q\geq 1$.

Let $w_{\nd-1}$ denote the weight defined by \eqref{wab} with $\alpha=\beta=\nd-1$, $\nd\ge 2$. If $p,q\geq 1$ and $s>0$, then the space $RA^s_{p,q}(\Rn)$ is isomorphic to the space $RA^s_{p,q}(\R,w_{\nd-1})$ that consists  of even functions belonging to $A^s_{p,q}(\R, w_{\nd-1})$, cf.  \cite[Thms. 3 and 9]{SSV}. 
\end{remark}
%%%%%%%%%%%%%
 
We recall that the embedding
\[\id_R: \rae(\Rn)\hookrightarrow \raz(\Rn)\]
is compact if, and only if,
\begin{equation}
  s_1-s_2>\nd\left(\frac{1}{p_1}-\frac{1}{p_2}\right)>0 \qquad \text{and} \qquad \nd>1 ,\label{comp-radial}
 \end{equation}
cf. \cite{SS}. Further properties of  spaces of radial functions, in particular Strauss type inequalities as well as the description of traces on real lines through the origin, can be found    in \cite{SS,SSV}.

\begin{theorem}\label{radial-nuc}
Let $1\le p_i,q_i\le \infty$, $s_i\in\real$, $i=1,2$ ($ p_i< \infty$ in the case of $F_{p,q}^s$ spaces). 
Then the embedding 
\[ \id_R : \rae(\Rn)\hookrightarrow \raz(\Rn)\] 
is nuclear if, and only if,
\begin{equation*} s_1-s_2>\nd\left(\frac{1}{p_1}-\frac{1}{p_2}\right)>1.
  \end{equation*}
\end{theorem}
 
\begin{proof}
 It is sufficient to prove the theorem for Besov spaces and large values of $s_1$ and $s_2$. The rest follows by the elementary embeddings between Besov and Triebel-Lizorkin spaces in the sense of \eqref{B-F-B} and the lift property for the scale of Besov spaces.  So we assume that $s_1\ge s_2> 0$. It was proved in \cite{SSV} that the space $RB^{s_i}_{p_i,q_i}(\R^{\nd})$ is isomorphic to the weighted space $RB^{s_i}_{p_i,q_i}(\R, w_{\nd-1})$, cf. Theorem 3 and Theorem 9 ibidem. So the embedding $\rbe(\Rn)\hookrightarrow \rbz(\Rn)$ is nuclear if, and only if,  $\rbe(\R, w_{\nd-1})\hookrightarrow \rbz(\R, w_{\nd-1})$ is nuclear. But the double-weighted situation can be reduced to the one-side weighted case, i.e., the last embedding is nuclear if, and only if, the embedding  $\rbe(\R, w_{\alpha})\hookrightarrow \rbz(\R)$ with $\alpha = (\nd-1)(1-\frac{p_1}{p_2})$ is nuclear. Now Theorem~\ref{wab-nuc} (one-dimensional with $\beta=\alpha$) implies that the embedding is nuclear if $s_1-s_2>\nd(\frac{1}{p_1}-\frac{1}{p_2})>1$.  

 Conversely, if the embedding  $\id_R: \rbe(\Rn)\hookrightarrow \rbz(\Rn)$ is nuclear, then it is compact. This implies $s_1-s_2>\nd(\frac{1}{p_1}-\frac{1}{p_2})>0$, see \eqref{comp-radial}, in particular, $p_1<p_2$. Moreover the nuclearity of $\id_R: \rbe(\Rn)\hookrightarrow \rbz(\Rn)$ is equivalent to the nuclearity of $\rbe(\R, w_{\alpha})\hookrightarrow \rbz(\R)$. 
 
Furthermore, the space 
\[ RB^{s}_{p,q}(\Rn, (t,\infty ))=\left\{f\in RB^s_{p,q}(\Rn): \supp f \subset \{x\in\rn: |x|\ge t\}\right\},\]
is isomorphic to the space 
\[ B^{s}_{p,q}(\R, w_{\nd-1},(t,\infty ))=\left\{f\in B^s_{p,q}(\R,w_{\nd-1}): \supp f\subset [t,\infty)\right\},\]
cf. \cite{K-L-S-S-0}.   So if the embedding  $\id_R: \rbe(\Rn)\hookrightarrow \rbz(\Rn)$ is nuclear, then the embedding 
$\id: B^{s_1}_{p_1,q_1}(\R, w_{\nd-1},(t,\infty ))\hookrightarrow B^{s_2}_{p_2,q_2}(\R, w_{\nd-1},(t,\infty ))$ is nuclear, too. But now  we can use  the wavelet expansions and arguments similar to that ones that were used  for the  global behaviour of the purely polynomial weight in Step 2 of the proof of Theorem~\ref{wab-nuc}. Thus, the one-dimensional version of Theorem~\ref{wab-nuc}, in particular \eqref{wab-n-1} with $\alpha=\beta=(\nd-1)(1-\frac{p_1}{p_2})$, lead to
\[
\frac{\beta}{p_1}=\frac{\nd-1}{p_1}\left(1-\frac{p_1}{p_2}\right) > \frac{1}{\tn(p_1,p_2)} = 1-\left(\frac{1}{p_1}-\frac{1}{p_2}\right)
\]
in view of \eqref{tongnumber} and $p_1<p_2$, and
\[
\delta=s_1-s_2-\frac{1}{p_1}+\frac{1}{p_2}> \frac{\alpha}{p_1}=\frac{\nd-1}{p_1}\left(1-\frac{p_1}{p_2}\right).
\]
This finally results in $s_1-s_2>\nd(\frac{1}{p_1}-\frac{1}{p_2})>1$, as desired.
\end{proof}

%%%%%%%%%%%%%%%%%%%5

\bibliographystyle{plain}

%}

\vfill\smallskip
{\small
\noindent\begin{minipage}[t]{0.42\textwidth}
Dorothee D. Haroske\\
Institute of Mathematics \\
Friedrich Schiller University Jena\\
07737 Jena\\
Germany\\[1ex]
~\\
%\vfill
{\tt dorothee.haroske@uni-jena.de}
\end{minipage}\hfill
\begin{minipage}[t]{0.45\textwidth}
Leszek Skrzypczak\\
Faculty of Mathematics \& Computer Science\\
Adam  Mickiewicz University\\
ul. Uniwersytetu Pozna\'nskiego 4\\
61-614 Pozna\'n\\
Poland\\[1ex]
{\tt lskrzyp@amu.edu.pl}
\end{minipage}\\[0ex]
}

%%%%%%%%%%%%%%%%%%%5

\end{document}